\documentclass{article}
\usepackage[T1]{fontenc}
\usepackage[latin9]{inputenc}
\usepackage{array}
\usepackage{amsmath}
\usepackage{amssymb}
\usepackage{amsthm}
\usepackage{graphicx}
\usepackage[plain,noend]{algorithm2e}
\usepackage[authoryear]{natbib}
\usepackage{times}
\usepackage{bm}
\usepackage{comment}
\usepackage{color}

\DeclareMathOperator*{\argmin}{\textsf{Argmin}}

\makeatletter

\providecommand{\tabularnewline}{\\}

\makeatletter
\newcommand{\vertiii}[1]{{\left\vert\kern-0.25ex\left\vert\kern-0.25ex\left\vert #1
\right\vert\kern-0.25ex\right\vert\kern-0.25ex\right\vert}}

\renewcommand{\algocf@captiontext}[2]{#1\algocf@typo. \AlCapFnt{}#2} 
\def\@algocf@capt@plain{top}
\renewcommand{\algocf@makecaption}[2]{%
  \addtolength{\hsize}{\algomargin}%
  \sbox\@tempboxa{\algocf@captiontext{#1}{#2}}%
  \ifdim\wd\@tempboxa >\hsize
    \hskip .5\algomargin%
    \parbox[t]{\hsize}{\algocf@captiontext{#1}{#2}}
  \else%
    \global\@minipagefalse%
    \hbox to\hsize{\box\@tempboxa}
  \fi%
  \addtolength{\hsize}{-\algomargin}%
}
\makeatother

\newtheorem{theorem}{Theorem}[section]
\newtheorem*{theorem*}{Theorem}
\newtheorem{assumption}{Assumption}
\newtheorem{remark}{Remark}
\newtheorem{definition}{Definition}[section]
\newtheorem{proposition}{Proposition}
\newtheorem{lemma}{Lemma}
\newtheorem*{lemma*}{Lemma}
\newtheorem*{assumption*}{Assumption}

\def\Bka{{\it Biometrika}}

\makeatother

\begin{document}

\title{Multiple Changepoint Estimation in High-Dimensional Gaussian Graphical Models}
\author{A. Gibberd and S. Roy} 

\maketitle
\begin{abstract}
We consider the consistency properties of a regularised estimator
for the simultaneous identification of both changepoints and graphical
dependency structure in multivariate time-series. Traditionally, estimation
of Gaussian Graphical Models (GGM) is performed in an i.i.d setting.
More recently, such models have been extended to allow for changes
in the distribution, but only where changepoints are known a-priori.
In this work, we study the Group-Fused Graphical Lasso (GFGL) which
penalises partial-correlations with an  L1 penalty while simultaneously
inducing block-wise smoothness over time to detect multiple changepoints.
We present a proof of consistency for the estimator, both in terms of changepoints, and the structure of the graphical models in each segment. 
\end{abstract}

\section{Introduction}

\subsection{Motivation}

Many modern day datasets exhibit multivariate dependance structure
that can be modelled using networks or graphs. For example, in social
sciences, biomedical studies, financial applications etc. the association
of datasets with latent network structures are ubiquitous. Many of
these datsets are time-varying in nature and that motivates the modelling
of dynamic networks. A network is usually characterised by a graph
$G$ with vertex set $V$ (the collection of nodes) and edge set $E$
(the collection of edges). We denote $G=(V,E)$. For example, in a
biological application nodes may denote a set of genes and the edges
may be the interactions among the genes. Alternatively, in neuroscience,
the nodes may represent observed processes in different regions of
the brain, and the edges represent functional connectivity or connectome.
In both situations, we may observe activity at nodes over a period
of time\textemdash{} the challenge is to infer the dependency network
and how this changes over time. 

In this paper we consider a particular type of dynamic network model
where the underlying dependency structure evolves in a piecewise fashion.
We desire to estimate multiple change-points where the network structure
changes, as well as the structures themselves. To this end, the task
is formulated as a joint optimization problem such that change-point
estimation and structure recovery can be performed simultaneously.
The particular class of networks we aim to estimate are encoded via
a multivariate Gaussian that has a piecewise constant precision matrix
over certain blocks of time. Specifically, we assume observations
$X^{(t)}=(X_{1}^{(t)},\ldots,X_{p}^{(t)})$ are drawn from the following
model

\begin{equation}
X_{t}\sim\mathcal{N}(0,\Sigma^{(k)}_0)\quad,\;t\in \{\tau_{k-1},\ldots,\tau_k\}\;,\label{eq:GFGL_model}
\end{equation}
where $t=1,\ldots,T$ indexes the time of the observed data-point, and
$k=1,\ldots,K+1$ indexes blocks seperated by changepoints $\{\tau_{k}\}_{k=1}^{K}$, at which
point the covariance matrix $\Sigma^{(k)}_0$ changes. The task is to assess, how well,
or indeed if, we can recover both the changepoint positions $\tau_{k}$
and the correct precision matrices $\Theta^{(k)}_0:=(\Sigma^{(k)}_0)^{-1}$. 

\subsection{Literature Review}

Although the topic of change-point estimation is well represented
in statistics (see \citet{Bai1997}, \citet{Hinkley1970}, \citet{Loader1996},
\citet{Lan2009}, \citet{Muller1992}, \citet{Raimondo1998} and references
therein), its application in high dimensional graphical models is
relatively unexplored. We will first review the literature on change-point detection before moving onto graphical model estimation and the intersection of these methodologies.

%
%
One of the
most popular methods of change-point detection is binary segmentation
\citep{Fryzlewicz2012,Cho2015}. The method hinges on dividing the
entire data into two segments based on a discrepancy measure and relating
the procedure on subsequent segments until there are no change-point
available. Usually a cumulative summation type test-statistic \citep{Page1954} is
used to compute the discrepancy based on possible two segments of
the data. Further approaches for multiple change-point
estimation in multivariate time series (not necessarily high dimensional)
include: SLEX (Smooth Localized Complex Exponentials), a complexity
penalized optimization for time series segmentation \citep{Ombao2005};
approaches which utilise a penalised gaussian likelihood via dynamic
programming \citep{Lavielle2006,Angelosante2011}; or penalized regression
utilising the group lasso \citep{Bleakley_2011}. More recently, work
on high dimensional time series segmentation includes: sparsified
binary segmentation via thresholding CUSUM statistics \citep{Cho2015};
change-point estimation for high dimensional time series with missing
data via subspace tracking \citep{Xie2013}, and projection based
approaches \citep{Wang2017}. Some recent works on high dimensional
regression with change-points are also worth mentioning here. For
instance \citet{Lee2016} consider high dimensional regression with
a possible change-point due to a covariate threshold and develop a
lasso based estimator for the regression coefficients as well as the
threshold parameter. \citet{Leonardi2016} extend this to multiple
change-point detection in high dimensional regression, the resultant
estimator can be used in conjunction with dynamic programming or in
an approximate setting via binary segmentation. The theoretical guarantees
for the two methods appear to be similar and estimate both the number
and locations of change-points, as well as the parameters in the corresponding
segments.

Moving on to graph estimation, in the static i.i.d. case there are two principle approaches to estimating
conditional independence graphs. Firstly, as suggested by \citet{Annals2000},
one may adopt a neighbourhood or \emph{local} selection approach where
edges are estimated at a nodewise level, an estimate for the network
is then constructed by iterating across nodes. Alternatively, one
may consider joint estimation of the edge structure across all nodes
in a \emph{global} fashion. In the i.i.d. setting a popular method
to achieved this is via the Graphical lasso \citep{Banerjee_2007},
or explicitly constrained precision matrix estimation schemes such
as CLIME \citep{Cai2011}. 

In the dynamic setting, one could consider extending static neighbourhood
selection methods, for instance utilising the methods of \citet{Lee2016,Leonardi2016}
to estimate a graph where each node may exhibit multiple change points.
The work of \citet{Roy2016} considers a neighbourhood selection approach
for networks in the presence of a single changepoint, while \citet{Kolar2012}
consider using the fused lasso \citep{Harchaoui2010} to estimate
multiple changepoints at the node level. In the global estimation
setting, \citet{Angelosante2011} proposed to combine the graphical
lasso with dynamic programming to estimate changepoints and graph
structures. Alternatively, one may consider smoothing variation at
an individual edge level via an $\ell_{1}$ penalty \citep{Conference2014,Monti2014},
or across multiple edges via a group-fused $\ell_{2,1}$ penalty \citep{Gibberd2017}. 

\subsection{Paper Contribution}

Unlike previous literature in high dimensional time series segmentation
or high dimensional regression settings with possible change-points,
our framework is in the context of a graphical model that changes
its underlying structure with time in a piecewise manner. We therefore
desire to detect jointly, both the change-points and the parameters
specifying the underlying network structure. To achieve such estimation,
we construct an M-estimator which jointly penalises sparsity in the
graph structure while additionally smoothing the structure over time.

In this paper we focus on describing multiple change-point impacting
the global structural change in the underlying Gaussian graphical
model. Whether a global, or local approach is appropriate will depend
on the application. For example, in certain biological applications
(proteins, biomolecules etc.) node-wise changes would be more appropriate,
whereas for genetic interaction networks, social interaction or brain
networks global structural changes may be of more interest. Additionally,
in many situations it is not trivial how we combine edge estimates
when performing a neighbourhood selection approach. Because edges
are estimated locally they may be inconsistent across nodes which
can make it difficult to interpret global changes.

To avoid these problems, we opt to perform global estimation and assume
the presence of changepoints which impact many of the edges together.
Specifically, this paper deals with analysing the Group-Fused Graphical
lasso (GFGL) estimator first proposed in \citet{Gibberd2017}. In
previous work it was demonstrated that empirically GFGL can detect
both changepoints and graphical structure while operating in relatively
high-dimensional settings. However, until now, the theoretical consistency
properties of the estimator have remained unknown. In this paper,
we derive rates for the consistent recovery of both changepoints and model structure via upper bounds on the errors: $\max_{k}|\hat{\tau}^{(k)}-\tau_{0}^{(k)}|$ and $\|\hat{\Theta}^{(k)}-\Theta_{0}^{(k)}\|_{\infty}$
under sampling from the model in Eq. \ref{eq:GFGL_model}. We discuss the obtained convergence rates with those already
described in the literature, such as changepoint
recovery rates with fused neighbourhood selection \citep{Kolar2012},
and node-wise binary-segmentation/dynamic programming routines \citep{Leonardi2016}. Our theoretical results demonstrate that we can asymptotically recover precision matrices, even in the presence of a-priori unknown changepoints. We demonstrate that error bounds of a similar form to those presented in the static i.i.d. graphical lasso case \citet{Ravikumar2011} can be obtained for sufficient $T$ and appropriate smoothing regularisation. Using such bounds, we attempt to quantify the price one pays for allowing non-homogeneous sampling and the presence of changepoints. To this end, our results indicate the efficiency of precision matrix estimation with GFGL may be limited for shorter sequences due to the influence and bias imposed by our assumed smoothing regularisers.

\subsection{Notation}

Throughout the paper we will use a kind of notational overloading
whereby $\Theta^{(k)}\in\mathbb{R}^{p\times p}$ refers to a \emph{block}
indexed precision matrix and $\Theta^{(t)}$ refers to a\emph{ time}-indexed
one, for $t=1,\ldots,T$. In general, each block $k=1,\ldots,K+1$ will contain multiple
time-points. It is expected that many $\Theta^{(t)}$ will be the
same, however, only a few $\Theta^{(k)}$ will be similar, a diagrammatic
overview of the notation can be found in Figure \ref{fig:tensor_notation}.

Vector norms $\ell_a$ are denoted $\|x\|_a$, similarly the application of vector norms to a vectorised matrix is denoted $\|X\|_a$. For instance, the maximum element of a matrix is denoted $\|X\|_\infty=\max_{i,j}|X_{ij}|$. Where the Frobenius norm is used, this is denoted $\|X\|_F$. Additionally, we utilise the operator norm $\vertiii{X}_2$ which denotes the largest singular values of $X$, and the $\ell_{\infty,\infty}$ norm $\vertiii{X}_{\infty}:=\max_{i}\sum_{j}|X_{ij}|$.

\begin{figure}
\begin{centering}
\includegraphics[width=1\columnwidth]{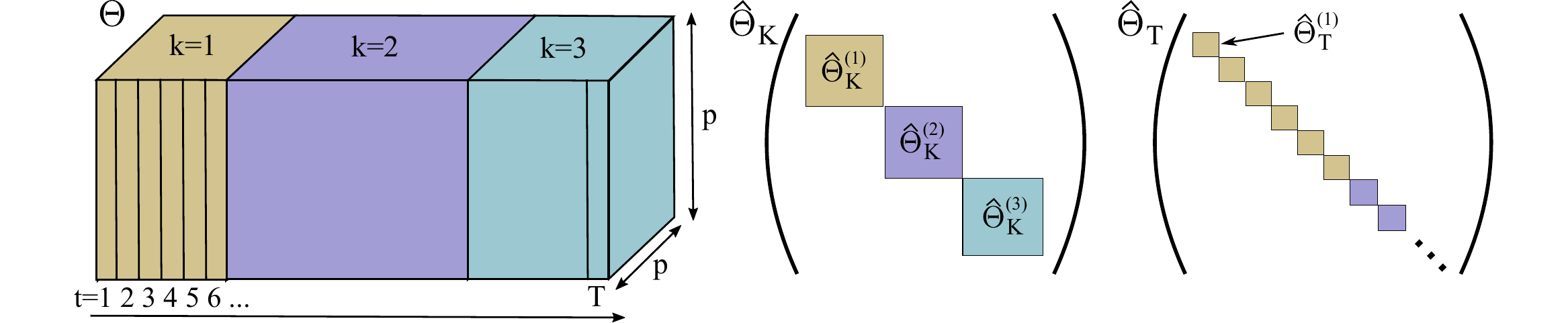}
\par\end{centering}
\caption{Diagramatic representation of matrix indexing and notation. \label{fig:tensor_notation}}

\end{figure}

\section{Estimator Formulation and Computation}

\subsection{The Group-Fused Graphical Lasso}

\label{sec:GFGL_sec2}

To simultaneously estimate multivariate Gaussian model structure alongside
changepoints, we propose to use the GFGL estimator first examined
in \citet{Gibberd2017}. This takes multivariate observations $x^{(t)}\in\mathbb{R}^{p}$
for $t=1,\ldots,T$ and estimates a set of $T$ precision matrices
represented in a block-diagonal model space $\hat{\Theta}_{T}\in\tilde{\mathbb{R}}^{Tp\times Tp}$.
In the following definition, we reference each diagonal matrix as
$\hat{\Theta}^{(t)}$ with the full model being described by $\{\hat{\Theta}^{(t)}\}_{t=1}^{T}$.

\begin{definition}{Group-Fused Graphical Lasso estimator}

Let $\hat{S}^{(t)}:=x^{(t)}(x^{(t)})^{\top}$ be the local empirical
covariance estimator. The GFGL estimator is defined as the M-estimator

\begin{equation}
\{\hat{\Theta}^{(t)}\}_{t=1}^{T}=\argmin_{\{U^{(t)}\succeq0\}_{t=1}^{T}}\left\{ l_{T}((U^{(t)},x^{(t)})_{t=1}^{T})+r_{T}((U^{(t)})_{t=1}^{T})\right\} \label{eq:GFGL_cost_function}
\end{equation}
where the loss-function and regulariser are respectively defined as

\begin{align}
l_{T}((U^{(t)},\hat{S}^{(t)})_{t=1}^{T}) & :=\sum_{t=1}^{T}\left\{ -\log\det(U^{(t)})+\mathrm{trace}(\hat{S}^{(t)}U^{(t)})\right\} \;,\label{eq:GFGL_loss_function}\\
r_{T}((U^{(t)})_{t=1}^{T}) & :=\lambda_1 \sum_{t=1}^{T}\sum_{i\ne j}|U_{i,j}^{(t)}|+\lambda_2 \sum_{t=2}^{T} \{\sum_{i,j=1}^{p}(U_{i,j}^{(t)}-U_{i,j}^{(t-1)})^{2}\}^{1/2} \;.\label{eq:GFGL_regulariser}
\end{align}
\end{definition}

Once the precision matrices have been estimated, changepoints are
defined as the time-points:
\[
\left\{ \hat{\tau}_{1},\ldots,\hat{\tau}_{\hat{K}}\right\} :=\left\{ t\:|\hat{\Theta}^{(t)}-\hat{\Theta}^{(t-1)}\ne0\right\} \;.
\]
While changepoints in the traditional sense are defined above, it
is convenient later on to consider the block separators $\hat{\mathcal{T}}:=\{1\}\cup\{\hat{\tau}_{1},\ldots,\hat{\tau}_{\hat{K}}\}\cup\{T+1\}$,
the added entries are denoted $\tau_{0}$ and $\tau_{\hat{K}+1}$.

\subsection{Comparison with Alternative Estimators}

The discussion in this paper studies the impact of implementing \emph{graph-wise}
smoothness constraints on precision matrix estimation through penalising
edge variation with the group Frobenius norm. This contrasts with
previous research that has predominently focussed on enforcing smoothness
constraints at an edge-by-edge level \citep{Kolar2011,Kolar2012,Monti2014,Danaher2013,Saegusa2016}.
For instance, one may replace the group-fused penalty in (\ref{eq:GFGL_regulariser})
with a linearly seperable norm such as the $\ell_{1}$ norm, i.e.
$\lambda_2 \sum_{t=1}^{T}\|\Theta^{(t)}-\Theta^{(t-1)}\|_{1}$. In the case
of $\ell_{1}$ fusing one may interpret the estimator as a convex
relaxation of the total number of jumps across all edges, we may define
the number of changepoints as $K_{0}:=(1/2)\sum_{t=2}^{T-1}\|\Theta_{\backslash ii}^{(t)}-\Theta_{\backslash ii}^{(t-1)}\|_{0}$,
where $\|\cdot\|_{0}$ counts the number of non-zero differences. 

Alternatively, the grouped nature of GFGL assumes that the graph structure
underlying a process changes in some sense systematically. Rather
than edges changing individually, we presume that many of the edges
change dependency structure at the same time. In this case, we consider
smoothness at the graph level by counting the number of non-zero differences
in a non-seperable norm, such that $K_{0}:=|\{\|\Theta_{\backslash ii}^{(t)}-\Theta_{\backslash ii}^{(t-1)}\|\ne0\}|$.
In this case, the estimator assumes that many of the active edges
may change at the same time and that changes in graphical structure
are therefore syncronised. 

\subsection{Numerical Optimisation}

Due to both the convexity and linear composition properties of the
cost function (\ref{eq:GFGL_cost_function}), optimisation may be
performed relatively efficiently. Although not the main focus of this
work, we here suggest an \emph{alternating directed method of moments
(ADMM)} algorithm which is suitable for the task of minimising the
GFGL cost. The algorithm presented here is a modified version of that
found in the original GFGL paper \citep{Gibberd2017}, a similar independently
developed algorithm that is applicable to this problem can be found
in the work of \citet{Hallac2017}.

Specifically, we propose to minimise an augmented Lagrangian for (\ref{eq:GFGL_cost_function})
where the variables are grouped into sets corresponding to; primal
$U$, auxiliary $V$, transformations of auxiliary variables $W$,
and dual variables $\{\mathcal{V}_{1},\mathcal{V}_{2},\mathcal{W}\}$:
\begin{align*}
 & \mathcal{L}(U,V_{1},V_{2},W,\mathcal{V}_{1},\mathcal{V}_{2},\mathcal{W}):=\sum_{t=1}^{T}\bigg(-\log\det(U^{(t)})+\mathrm{tr}(U^{(t)}S^{(t)})\bigg)\ldots\\
 & +\lambda_1 \sum_{t=1}^{T}\|[V_{1}^{(t)}]_{\backslash ii}\|_{1}+\lambda_2 \sum_{t=2}^{T}\|W^{(t)}\|_{F}+\frac{\gamma_{V_{1}}}{2}\bigg(\sum_{t=1}^{T}\|U^{(t)}-V_{1}^{(t)}+\mathcal{V}_{1}^{(t)}\|_{F}^{2}-\|\mathcal{V}_{1}^{(t)}\|_{F}^{2}\bigg)\ldots\\
 & +\frac{\gamma_{V_{2}}}{2}\bigg(\sum_{t=1}^{T-1}\|U^{(t)}-V_{2}^{(t)}+\mathcal{V}_{2}^{(t)}\|_{F}^{2}-\|\mathcal{V}_{2}^{(t)}\|_{F}^{2}\bigg)\ldots\\
 & +\frac{\gamma_{W}}{2}\bigg(\sum_{t=2}^{T}\|(V_{1}^{(t)}-V_{2}^{(t-1)})-W^{(t)}+\mathcal{W}^{(t)}\|_{F}^{2}-\|\mathcal{W}^{(t)}\|_{F}^{2}\bigg)\;.
\end{align*}
The quantities $\gamma_{V_{1}},\gamma_{V_{2}},\gamma_{W}$ are algorithmic
tuning parameters which assign weight to deviation between the auxilary
and primal parameters, for instance via terms like $(\gamma_{V_{1}}/2)\|U^{(t)}-V_{1}^{(t)}\|_{F}$.
In practice, the algorithm appears numerically stable for equal weighting
$\gamma_{V_{1}}=\gamma_{V_{2}}=\gamma_{W}=1$. The ADMM algorithm
proceeds to perform dual ascent iteratively (and sequentially) minimising
$\mathcal{L}\{\cdot\}$ for $U,V_{1},V_{2},W$ and then updating the
dual variables. Pseudocode for these operations are given in Algorithm
\ref{alg:ADMM_pseudo}, a more detailed schema can be found in the
Supplementary Material. 

\begin{algorithm}[H]
\SetKwInput{KwInit}{Init}
\KwIn{Data and regulariser parameters: $X, \lambda_1, \lambda_2 $, convergence thresholds: $t_p,\: t_d$}
\KwOut{Precision matrices: $\{ \hat{\Theta}^{(t)} \}_{t=1}^T$ and  changepoints: $  \{ \hat{\tau}_1,\ldots ,\hat{\tau}_{\hat{K}} \}$}
\KwInit{$U^{(t)}=V^{(t)}=W^{(t)}=I\in\mathbb{R}^{p\times p}$}

\While{$\epsilon_p > t_p$, $\epsilon_d > t_d$}{
Update Primal Estimate:
$U^{(t)}_{n+1} = \arg\min_{U}\mathcal{L}\{\}$ - Eigen-decomposition\;
Update Auxilary Estimates:\\
$V_{1;n+1}^{(t)} = \arg\min_{V_1}\mathcal{L}\{\cdot\}$ - Soft-threshold\;
$V_{2;n+1}^{(t)}= \arg\min_{V_2}\mathcal{L}\{\cdot\}$ - Matrix multiplication\;
$W^{t}_{n+1} = \arg\min_{W}\mathcal{L}\{\cdot\}$ - Group soft-threshold differences\;
Update Dual variables: c.f. $\mathcal{V}_{1;n+1}^{(t)}=\mathcal{V}_{1;n}^{(t)}+U_{n+1}^{t}-V_{1;n}^{(t)}$ \;
}
\KwRet{$\{ \hat{\Theta}^{(t)} \}$}

\caption{ADMM algorithm for the minimisation of the GFGL cost function. \label{alg:ADMM_pseudo}}
\end{algorithm}

In particular, due to seperability of the ADMM updates, this algorithm
can be trivially parallelised across time-points and obtains an iteration
time complexity of order $\mathcal{O}(p^{3}T)$ in terms of the number
of data-points and dimension $p$. This contrasts favourably with
dynamic programming, which has a naiive cost $\mathcal{O}(T^{2})$
\citep{Angelosante2011}, or $\mathcal{O}(T\log(T))$ in the pruned
case \citep{Killick2012}, the binary segmentation scheme of \citet{Leonardi2016}
also obtains this latter rate. The cubic cost with regards to dimension
is due to the eigen-decomposition step. Potentially, this may be reduced
if some form of covariance pre-screeneing can be achieved, for instance
as suggested in \citet{Danaher2013} it may be possible to break the
$p$ dimensional problem into a set of independent $p_{1},\ldots,p_{M}$
dimensional problems where $\sum_{m=1}^{M}p_{m}=p$ . We leave such
investigations as directions for further work.

\subsection{Optimality conditions}

To assess the statistical properties of the GFGL estimator, we first need to derive a set of conditions which all minimisers of the cost function obey. We connect this condition to the sampling of the model under (\ref{eq:GFGL_model}) by the quantity $\Psi^{(t)}:=\hat{S}^{(t)}-\Sigma^{(t)}$
which represents the difference between the ground-truth covariance and the empirical
covariance matrix $X^{(t)}(X^{(t)})^{\top}$. 

Since we are dealing with a fused regulariser, it is convinient to introduce
a matrix $\Gamma^{(t)}$ corresponding to the differences in precision matrices. For the first step let $\Gamma^{(1)}=\Theta^{(1)}$, then for $t=2,\ldots,T$, let $\Gamma^{(t)}=\Theta^{(t)}-\Theta^{(t-1)}$. The sub-gradients for the non-smooth portion of the cost function are denoted respectively as $\hat{R}_{1}^{(t)},\hat{R}_{2}^{(t)} \in \mathbb{R}^{p\times p}$, for the $\ell_{1}$ and the
group-smoothing penalty. In full, these can be expressed as
\[
\hat{R}_{1;(i,j)}^{(t)}=\begin{cases}
\mathrm{sign}(\sum_{s\le t}\Gamma_{i,j}^{(s)}) & \mathrm{if}\;\sum_{s\le t}\Gamma_{i,j}^{(s)}\ne0\\{}
[-1,1] & \mathrm{otherwise}
\end{cases}\; ; \;\hat{R}_{2}^{(t)}=\begin{cases}
\frac{\hat{\Gamma}^{(t)}}{\|\hat{\Gamma}^{(t)}\|_{F}} & \mathrm{if}\;\hat{\Gamma}^{(t)}\ne 0\\
\mathcal{B}_{F}(0,1) & \mathrm{otherwise}
\end{cases}\;,
\]
where $\mathcal{B}_{F}(0,1)$ is the Frobenius unit ball. 

\begin{proposition}

\label{prop:optimality_conditions}

The minimiser $\{\hat{\Theta}^{(t)}\}_{t=1}^{T}$
of the GFGL objective satisfies the following

\[
\sum_{t=l}^{T}\left\{ (\Theta^{(t)})^{-1}-(\hat{\Theta}^{(t)})^{-1} \right \}-\sum_{t=l}^{T}\Psi^{(t)}+\lambda_1\sum_{t=l}^{T}\hat{R}_{1}^{(t)}+\lambda_2\hat{R}_{2}^{(l)}=0\;,
\]
for all $l\in[T]$ and $\hat{R}_{2}^{(1)}=\hat{R}_{2}^{(T)}=0$. 

\end{proposition}

\section{Theoretical Properties of the GFGL Estimator}

\subsection{Consistent Changepoint Estimation}

\label{sec:Theory_main}

We here present two complimentary results for changepoint consistency with the GFGL estimator. The first represents a standard asymptotic setting where $p$ is fixed and $T\rightarrow\infty$. The second result utilises a different concentration bound to constrain changepoint estimation error even in high-dimensions where $p>T$, in the doubly asymptotic setting $(p,T)\rightarrow\infty$. Further to the Gaussian sampling model in (\ref{eq:GFGL_model}), the high-dimensional result holds for any collection $(X^{(1)},\ldots,X^{(T)})$ where each $X_i^{(t)} / (\Sigma_{ii}^{(t)})^{1/2}$ for $t=1,\ldots,T$ and $i=1,\ldots,p$ is sub-Gaussian with parameter $\sigma$.

The changepoint consistency result presented here will take the form
of an upper bound on the maximum error of estimating a changepoint.
Let $\{\delta_{T}\}_{T\ge1}$ be a non-increasing positive sequence
that converges to zero as $T\rightarrow\infty$. This quantity should
converge at a rate which ensures an increasing absolute quantity $T\delta_{T}\rightarrow\infty$
as $T\rightarrow\infty$. The target of our results is to bound the maximum error to an ever decreasing proportion of the data, i.e. $\max_k|\hat{\tau}_k-\tau_k|/T\le \delta_T$. To establish a bound, we consider the setting where the minimum true distance between changepoints $d_{\min}:=\min_{k\in[K+1]}|\tau_{k}-\tau_{k-1}|$ increases with $T$, for simplicity let us assume this is bounded as a proportion $\gamma_{\min}<d_{\min}/T$. 

In order to control the variance in the underlying model it is required
to introduce several assumptions on the generating process:
\begin{assumption}{Bounded Eigenvalues}
\label{ass:bounded_eigenvalues}

There exist a constant $\phi_{\max}<\infty$ that
gives the maximum eigenvalues of the true covariance matrix
(across all blocks) such that $\phi_{\max}=\max_{k\in[K+1]}\{\Lambda_{\max}(\Sigma^{(k)})\}$.

\end{assumption}

\begin{assumption}{Bounded jump size}

\label{ass:finite_jumps_GFGL}

There exists a constant $M>0$ such that the difference between any
two blocks is bounded by a constant $\max_{k,k'\in[K+1]}\|\Sigma^{(k')}-\Sigma^{(k)}\|_{F}\le M $.
Additionally, the jumps are lower bounded according to $\min_{k\in[K+1]}\| \Sigma^{k}-\Sigma^{k-1}\|_F \ge \eta_{\min}$.

\end{assumption}

In addition to controlling variation in the sampling, we also need to ensure the regularisers $\lambda_1,\lambda_2$ are set appropriately. Generally one can increase or decrease the smoothing regulariser $\lambda_2$ to respectively decrease or increase the number of estimated changepoints However, in practice, we do not know a-priori the true number of changepoints in the model. Pragmatically, we may adopt a cross-validation strategy to choose the regularisers from data. However, in the theoretical setting, we assume that the regularisers are such that the correct number of changepoints are estimated.

\begin{assumption}{Appropriate Regularisation}
\label{ass:appropriate_reg}

Let $\lambda_{1}$ and $\lambda_{2}$ be specified such the solution of the GFGL problem (\ref{eq:GFGL_cost_function}) results in $|\hat{K}|=K$ changepoints. Furthermore, assume that $T$ is large enough such that $\beta_{1}:=(\eta_{\min}\gamma_{\min}T) \lambda_{2}^{-1}>32$,
$\beta_{2}:=\eta_{\min}\lambda_{1}^{-1}(p(p-1))^{-1/2}>8$,
and $\beta_{3}:=(\eta_{\min}T\delta_{T})\lambda_{2}^{-1}>3$. 
\end{assumption}

\begin{theorem}{Changepoint Consistency (Standard-dimensional asymptotics)}
\label{thm:cp_consistency_sd}

Given Assumptions \ref{ass:bounded_eigenvalues},\ref{ass:finite_jumps_GFGL},\ref{ass:appropriate_reg}. For a finite, but large $T$ such that $2<p\le T\delta_T$ and the jump size is lower bounded according to $\eta_{\min}>20\phi_{\max}e^{-1/2}p^{1/2}\sqrt{\log(T\delta_T)/T\delta_T}$, then the maximum changepoint error of GFGL (under sampling Eq. \ref{eq:GFGL_model}) is bounded according to probability
\[
P(\max_{k\in[K]}|\tau_{k}-\hat{\tau}_{k}|\le T\delta_{T})\ge1-C_{K}2(T\delta_T)^{-p/2}\rightarrow 1 \quad \mathrm{as}\quad T\rightarrow \infty\;,
\]
where $C_{K}=K(K^{2}2^{K+1}+4)$. 

\end{theorem}

\begin{theorem}{Changepoint-consistency (High-dimensional asymptotics)}
\label{thm:cp_consistency_hd}

Given Assumptions \ref{ass:bounded_eigenvalues},\ref{ass:finite_jumps_GFGL},\ref{ass:appropriate_reg}. If $\eta_{\mathrm{min}}>5c_{\sigma} 2^5 p\sqrt{\log (p^{\beta/2})/T\delta_T }$, where $c_\sigma=(1+4\sigma^2)\max_{ii,k}\{ \Sigma_{0;ii}^{(k)}\}$ and $\eta_{\min}\in(0,2^4c_{\sigma}p)$, then 
\[
P(\max_{k\in[K]}|\tau_{k}-\hat{\tau}_{k}|\le T\delta_{T})\ge1-C_{K}1/p^{\beta-2}\rightarrow 1\quad\mathrm{as}\quad (T,p)\rightarrow \infty\;.
\]
\end{theorem}
\begin{proof}
The proofs of the above follow a similar line of
argument as used in \citet{Harchaoui2010,Kolar2012}. However, several
modifications are required in order to allow analysis with the Gaussian
likelihood and group-fused regulariser. Additionally, for Theorem \ref{thm:cp_consistency_hd} we investigate the use of concentration bounds that operate in the setting where $p>T\delta_T$. Full details can be found in the Appendix (ref).
\end{proof}

The results demonstrate that asymptotically changepoint error can be constrained to a decreasing fraction $\delta_T \rightarrow 0$ of the time-series. Indeed, this occurs with increasing probability when estimation is performed with an increasing number of data-points. In both results, the minimal jump size must be sufficient in order to consistently detect changes. In the standard-dimensional setting we require $\eta_{\min} = \Omega(p^{1/2} \sqrt{\log (T\delta_T)/(T\delta_T)})$, and therefore asymptotically we can recover increasingly small changes.
In the doubly asymptotic setting of Theorem \ref{thm:cp_consistency_hd} we have $\eta_{\min}=\Omega (p\sqrt{\log (p^{\beta/2})/(T\delta_T)})$, thus we still require $\eta_{\min}$ to grow with $p$. One should note that the high-dimensional bound is only guaranteed for $\eta_{\min}\in(0,2^4c_{\sigma}p)$ which stipulates a minimal sample size $T\delta_T=\Omega(\log(p^{\beta/2}))$. 

\subsection{Consistent Graph Recovery}

\label{sec:graph_recovery}

One of the key properties of GFGL is that it simultaneously estimates
both the changepoint and model structure. In this section we will
turn our eye to the estimation of model structure in the form of
the precision matrices between changepoints. In particular, we consider
that a set of $\hat{K}=K$ changepoints have been identified as per Assumption \ref{ass:appropriate_reg} and Theorem \ref{thm:cp_consistency_hd}. In the previous
section we developed theory for the recovery of changepoints, for
a given (large) number of data-points we demonstrated that the changepoint
error can be bounded within a region $T\delta_{T}\rightarrow0$. We here assume that such a bound holds, and develop the theory in the high-dimensional setting analogous to Theorem \ref{thm:cp_consistency_hd}.

\begin{assumption}{Changepoint error is small}

\label{ass:bounded_cp_error}

Consider that the number of changepoints is estimated correctly, and the changepoint consistency event
\[
E_{\tau}:=\left\{ \max_{k\in[B]}|\hat{\tau}_{k}-\tau_{k}|\le T\delta_{T}\right\}
\]
holds with some increasing probability $f_{\tau}(p,T)\rightarrow1$
as $(p,T)\rightarrow\infty$.

\end{assumption}

A key advantage of splitting the changepoint and model-estiamtion
consistency arguments as we do here, is that we can consider a simplified
model structure such that the GFGL estimator may be parameterised
in terms of a $K+1$ block-diagonal matrix $\Theta_{K}\in\tilde{\mathbb{R}}^{Bp\times Bp}$.
Conditional on segmentation, we do not need to deal with the fact
that the model may be arbritrarily mis-specified (this is bounded
by Assumption \ref{ass:bounded_cp_error}). As such, in this section
the dimensionality of the model space is fixed with respect to an
increasing number of time-points. The following results demonstrate,
that as expected, gathering increasing amounts of data relating to
a fixed number of blocks allows us to identify the model with increasing
precision.

To demonstrate model recovery, let us define a set of pairs $\mathcal{M}_{k}$
which indicate the support set of the true model in block $k$ and
its compliment $\mathcal{M}_{k}^{\perp}$ as
\[
\mathcal{M}_{k}=\{(i,j)\:|\:\Theta_{0;i,j}^{(k)}\ne0\}\quad\mathrm{and}\quad\mathcal{M}_{k}^{\perp}=\{(i,j)\:|\:\Theta_{0;i,j}^{(k)}=0\}\;.
\]
The recovery of the precision matrix sparsity pattern in true block $l$ from estimated block $k$ can be monitored by the sign-consistency event defined:
\[
E_{\mathcal{M}}(\hat{\Theta}^{(k)};\Theta_{0}^{(l)}):=\{\mathrm{sign}(\hat{\Theta}_{ij}^{(k)})=\mathrm{sign}(\Theta_{0;ij}^{(l)})\;\forall i,j\in\mathcal{M}_{l}\}\;.
\]

In order to derive bounds on model recovery, one must make assumptions
on the true structure of $\Theta_{0}$. Such conditions
are often referred to as \emph{incoherence} or \emph{irrepresentability}
conditions. In the setting of graphical structure learning, these
conditions act to limit correlation between edges and restrict the
second order curvature of the loss function. In the case where we
analyse GFGL under Gaussian sampling the Hessian $\Gamma_{0}\equiv\nabla_{\Theta}^{2}L(\Theta)|_{\Theta_{0}}$
relates to the Fisher information matrix such that $\Gamma_{0;(j,k)(l,m)}=\mathrm{Cov}(X_{j}X_{k},X_{l}X_{m})$.
Written in this form we can understand the Fisher matrix as relating
to the covariance between \emph{edge variables} defined as $Z_{(i,j)}=X_{i}X_{j}-\mathbb{E}[X_{i}X_{j}]$
for $i,j\in\{1,\ldots,p\}$.

\begin{assumption}{Incoherence Condition} 

\label{ass:incoherence}

Let $\mathcal{M}$ denote the set of components relating to true edges
in the graph and $\mathcal{M}^{\perp}$ (for block $k$) its compliment. For example,
$\Gamma_{0;\mathcal{M}\mathcal{M}}^{(k)}$ refers to
the sub matrix of the Fisher matrix relating to edges in the true
graph. Assume that for each $k=1,\ldots,K+1$ there exists some $\alpha_{k}\in(0,1]$
such that 
\[
\underset{e\in\mathcal{M}^{\perp}}{\max}\|\Gamma_{0;e\mathcal{M}}^{(k)}(\Gamma_{0;\mathcal{M\mathcal{M}}}^{(k)})^{-1}\|_{1}\le(1-\alpha_{k})\,.
\]

\end{assumption}

In the multivariate Gaussian case we have $\max_{e\in\mathcal{M}^{\perp}}\|\mathbb{E}[Z_{e}Z_{\mathcal{M}}^{\top}]\mathbb{E}[Z_{\mathcal{M}}Z_{\mathcal{M}}^{\top}]^{-1}\|_{1}\le(1-\alpha_k)$
for each $k$, in the theory below we denote and track $\alpha = \min_k \{\alpha_k\}$. One can therefore interpret the incoherence condition
as a statement on the correlation between edge variables which are
outside the model subspace $Z_{(i,j)}$ such that $(i,j)\not\in \mathcal{E}$,
with those contained in the true model $(i,j)\in \mathcal{E}$. In practice,
this sets bounds on the types of graph and associated covariance structures
which estimators such as graphical lasso can recover (see the discussion
Sec. 3.1.1 \citet{Ravikumar2011} and \citet{Meinshausen2008}). The theory presented here can be seen as an extension of \citet{Ravikumar2011} to non-stationary settings. Similarly to their original analysis we will track the size of the operator norms $ \vertiii{\Sigma_0^{(k)}}_\infty:=\max_i \sum_{j=1}^p |\Sigma_{0;ij}^{(k)}|$ and $\vertiii{\Gamma_0^{(k)}}_\infty$, we simplify our analysis by tracking the upper bound $K_{\Sigma_0}:=\max_k \vertiii{\Sigma_0^{(k)} }_\infty$ and $K_{\Gamma_0}:=\max_k{\vertiii{\Gamma_0^{(k)} } }_\infty$.

When using GFGL there will generally be an error associated with the
identification of changepoints and as such the estimated and ground-truth
blocks do not directly align. With this in mind, the model consistency
proof we present does not necessarily compare the $k$th estimated
block, to the $k$th ground-truth block. Instead, the proof is constructed
such that the structure in an estimated block $k\in[\hat{B}]$ is compared
to the ground-truth structure in block $l$ such that the blocks $k$
and $l$ maximally overlap with respect to time. Notationally, let $\hat{n}_k=\hat{\tau}_k-\hat{\tau}_{k-1}$ and $\hat{n}_{lk}=| \{ \hat{\tau}_{k-1},\ldots,\hat{\tau}_{k} \} \cap \{ \tau_{l-1},\ldots,\tau_{l} \} |$. The maximally overlapping block is then defined as 
$
k_{\max}=\arg\max_{l} \{ \hat{n}_{lk}  \}
$.

\begin{theorem}{Bounds on the estimation error}
\label{thm:bound_on_estimation error}

Consider the GFGL estimator with Assumptions (\ref{ass:bounded_cp_error},\ref{ass:incoherence}). Assume $\lambda_1=16\alpha^{-1}\epsilon$, $\lambda_2=\rho\lambda_1$ for some finite $\rho>0$ and $\epsilon>2^4\sqrt{2}c_{\sigma_\infty}\sqrt{\log(4p^{\beta}) /\gamma_{\min} T}$ where $\beta>2$. Let $d$ be the maximum degree of each node in the true model, and define  
\begin{equation}
\label{eq:min_sample_condition_infty_error_bound}
v_{\mathcal{C}}=6\{1+16\alpha^{-1}(1+2\hat{n}_k^{-1}\rho)\}d\max\{ K_{\Sigma_{0}}K_{\Gamma_{0}},K_{\Gamma_{0}}^2 K_{\Sigma_{0}}^3 \}\;.
\end{equation}
Then for $T\ge 2^9 \gamma_{\min}^{-1}\max\{1/8c_{\sigma_\infty},v_{\mathcal{C}}\}^2 c^2_{\sigma_\infty}\log(4p^\beta)$, we have
\begin{equation}
\label{eq:precision_error_bound}
\|\hat{\Theta}^{(k)}-\Theta^{(k_{\max})}_0\|_\infty\le 2K_{\Gamma_{0}}\{1+16\alpha^{-1}(1+2\rho \hat{n}_k^{-1}) \}\epsilon\;,
\end{equation}
in probability greater than $1-(1/p^{\beta-2}+f_\tau(p,T))$.

\end{theorem}

\begin{theorem}{Model-selection consistency}
\label{thm:sign_consistency}

In addition to the assumptions in Theorem \ref{thm:bound_on_estimation error}. Let $\theta_{\min}^{(k)} := \min_{ij}|\Theta^{(k)}_{0;ij}|$ for all $(i,j)\in\mathcal M_k$ and for each $k=1,\ldots,B$. Let
\[
v_{\theta} = 2K_{\Sigma_0}\{1+16\alpha^{-1}(1+2\rho\hat{n}_k^{-1})\}\theta_{\min}^{-1}\;.
\]
If $T\ge 2^9\gamma_{\min}^{-1}\max\{(8c_{\sigma_\infty})^{-1},v_\mathcal{C},v_{\theta} \}^2c_{\sigma_\infty}^2\log(4p^\beta)$ then GFGL attains sign-consistency 
\[
P\{E_{\mathcal{M}}(\hat{\Theta}^{(k)};\Theta_{0}^{(k_{\max})})\}\ge 1- \{1/p^{\beta-2}+f_\tau(p,T)\}\;,
\]
with probability tending to one as $(p,T)\rightarrow\infty$.
\end{theorem}

\begin{proof}
Theorem \ref{thm:bound_on_estimation error} is obtained utilising a primal-dual-witness approach conditional on the event $E_\tau$. This follows a similar argument to that used in \citet{Ravikumar2011}, but requires modifications due to the smoothing regulariser in GFGL. Theorem \ref{thm:sign_consistency}, is a corollary of Theorem \ref{thm:bound_on_estimation error} on the condition that the true entries in the precision matrix are sufficiently large. Full details of the proof are found in Appendix .
\end{proof}

The above bounds suggest that indeed, if regularisation is appropriately set one can obtain a consistent estimate of the precision matrices from GFGL. However, there are several important insights we can take from the results. Firstly, we clearly see the effect of the smoothing regulariser in Eq. \ref{eq:precision_error_bound}. In particular, a larger $\rho$ will result in a larger upper bound for the error. In the analagous results from the i.i.d. graphical lasso case \citet{Ravikumar2011}, the bound here is of a form $2K_{\Gamma_0}(1+8\alpha^{-1})\epsilon$. In fact, our results suggest that the additional error in the precision matrix is a function of the ratio $\lambda_2(\lambda_1\hat{n}_k)^{-1}$, if $\lambda_2/\lambda_1$ does not grow faster than $\hat{n}$ then estimation consistency can be achieved. If we assume, for example that $\lambda_1=\mathcal{O}(\sqrt{\log(p^\beta)/T})$ and $\hat{n}_k =\mathcal{O}(T)$, then this gives $\lambda_2$ the flexibility to grow with increasing $T$, but at a rate $\lambda_2=\mathcal{O}(\sqrt{T\log p^\beta})$. We note that under such scaling it is possible to satisfy the conditions of Assumption \ref{ass:appropriate_reg} thus both changepoint and structure estimation consistency is achieved.

\section{Discussion}

\subsection{A note on minimum detectable jumps}
When compared to the neighbourhood based estimators, GFGL considers changepoints
at the full precision matrix scale as opposed to seperately for each node. One might therefore expect
that the minimum jump size required in our result (Theorem 1) is greater
than that utilised in neighbourhood selection case (c.f. \citet{Kolar2012}).
For example, in the neighbourhood selection case, one may consider
the analogous quantity $\eta_{\min}^{\mathrm{NS}}(a):=\min_{k\in[K]}\|\Sigma_{a,\cdot}^{(k+1)}-\Sigma_{a,\cdot}^{(k)}\|_{2}$,
for nodes $a=1,\ldots,p$. Summing over nodes, the neighbourhood selection
jump size can now be related to the mixed (group) norm $\|\Sigma^{(k+1)}-\Sigma^{(k)}\|_{2,1}=\sum_{a}\|\Sigma_{a,\cdot}^{(k+1)}-\Sigma_{a,\cdot}^{(k)}\|_{2}$.
Furthermore, if the smallest jump occurs at the same block for each
neighbourhood, i.e. $\arg\min_{k\in[K]}\|\Sigma_{a,\cdot}^{(k+1)}-\Sigma_{a,\cdot}^{(k)}\|_{2}=\arg\min_{k\in[K]}\|\Sigma_{a',\cdot}^{(k+1)}-\Sigma_{a',\cdot}^{(k)}\|_{2}$
for all $a\ne a'\in[p]$ , then $\sum_{a}\eta_{\min}^{\mathrm{NS}}(a)=\min_{k}\|\Sigma^{(k+1)}-\Sigma^{(k)}\|_{2,1}$.
Using the inequality (for $x\in\mathbb{R}^{n}$) $\|x\|_{2}\le\|x\|_{1}\le\sqrt{n}\|x\|_{2}$,
the jumps as measured through the group-norm can be related to those
measured in a Frobenius sense, such that $\eta_{\min}\le\sum_{a}\eta_{\min}^{\mathrm{NS}}(a)\le\sqrt{p}\eta_{\min}$.

Thus, even though the minimum jump size in the GFGL case is greater,
i.e. $\eta_{\min}>\eta_{\min}^{\mathrm{NS}}(a)$, it is not proportionally
greater when one considers summing over nodes. In our analysis it
should be noted that consistent recovery of changepoints requires
a tradeoff between the minimum jump-size $\eta_{\min}$ and the amount
of data $T$. For example, a smaller minimum jump-size will generally
require more data; as expected it is harder to detect small jumps.
The relation $\eta_{\min}\le\sum_{a=1}^{p}\eta_{\min}^{\mathrm{NS}}(a)$
suggests that the minimum jump-size at a graph-wide (precision matrix
wide) level is proportionally smaller when measured in the Frobenius norm, than at a node-wise level. As a result, for equivalent scaling of
$\eta_{\min}$ and $\eta_{\min}^{\mathrm{NS}}$ the graph-wide GFGL
method will be able to detect smaller (graph-wide) jumps with an equivalent
level of data. Conversely, if the jumps one is interested in occur
at the neighbourhood level the neighbourhood based method would be
more appropriate, although this is generally not the case with the
block-constant GGM model (\ref{eq:GFGL_model}).

\subsection{Consistency with increasing dimensionality}

The work of \cite{Kolar2012} presents results in a standard asymptotic $T\rightarrow \infty$ setting. As such, they do not assess how fast changepoint consistency is achieved in relation to the number of data-streams $p$. In the case of GFGL, we give an upper bound on $P(E_\tau)$ (Theorem \ref{thm:cp_consistency_sd}) that directly relates to the dimensionality of the problem. Of particular note, is that the convergence rate increases with the number of data-streams used. This aligns with what one may intuitively expect, i.e. if changepoints are shared across most data-streams, then increasing the number of data-streams will increase the "signal" associated with changepoints. We may thus improve changepoint detection performance by performing joint inference for changepoints.

In our high-dimensional analysis (Theorems \ref{thm:cp_consistency_hd}, \ref{thm:bound_on_estimation error},\ref{thm:sign_consistency}) we consider what happens when $p>T\delta_T$. Our upper bound for changepoint error has a requirement that $\eta_{\min}=\Omega(p^{-1}\sqrt{\log(p^\beta/2)/T\delta_T})$. We note that this lower bound allows for the calculation of a minimal $T$ for the bound to hold. However, one must also take care that conditions in Assumption \ref{ass:appropriate_reg} are met. While $\lambda_1$ can be seen somewhat as a free parameter in the proof of changepoint consistency, it plays a key role in bounding the estimation error of the block-wise precision matrices. Specifically, Theorem \ref{thm:bound_on_estimation error} mandates $\lambda_1=\Omega(\sqrt{\log (p^\beta)/T})$. We note, that this lower bound is generally sufficient to meet Assumption \ref{ass:appropriate_reg}, for changepoint consistency, which requires $\lambda_1=\mathcal{O}(\eta_{\min}/p)=\mathcal{O}(\sqrt{\log(p^(\beta/2))/T\delta_T})$ when $\eta_{\min}$ is guided by the lower bound in Theorem \ref{thm:cp_consistency_hd}.

\appendix


\section{Numerical Optimisation and optimality conditions}

\subsection{Multi-block ADMM algorithm for GFGL}

Algorithm \ref{alg:ADMM_full} provides extended detail relating to
the pseudo-code presented in the main paper. The particular steps
for each of the proximity updates (minimisation steps in the augmentd
lagrangian) are given. An alternative ADMM algorithm for solving the
GFGL problem is described in \citet{Hallac2017} and \citet{Gibberd2017}.

\begin{algorithm}[h]
\SetKwInput{KwInit}{Init}
\KwIn{Data and regulariser parameters: $\boldsymbol{X}, \lambda, \rho $, convergence thresholds: $t_p,\: t_d$}
\KwOut{Precision matrices: $\{ \hat{\boldsymbol{\Theta}}^{(t)} \}_{t=1}^T$ and  changepoints: $  \{ \hat{\tau}_1,\ldots ,\hat{\tau}_{\hat{K}} \}$}
\KwInit{$U^{(t)}=V^{(t)}=W^{(t)}=I\in\mathbb{R}^{p\times p}$}

\While{$\epsilon_p > t_p$, $\epsilon_d > t_d$}{
\For{t=1,\ldots,T}{
$P^{(t)}=\begin{cases}V_{1;n}^{(t)}-\mathcal{V}_{1;n}^{(t)} & t=T\\(V_{1;n}^{(t)}-\mathcal{V}_{1;n}^{(t)})+(V_{2;n}^{(t)}-\mathcal{V}_{2;n}^{(t)}) & \mathrm{otherwise}\end{cases}$\;
Perform eigen-decomposition: $\{ \eta^{\top}_n,L\} = \mathrm{eig}\{ (\hat{S}^{(t)}-P^{(t)})\}$ \;
Update eigenvalues, for each $i=1,\ldots,p$ : $\eta_{i;n+1} = -(\eta_{i;n} -\sqrt{\eta_{i;n}^{2} +8})/4$ \;
$U^{(t)}_{n+1} = L\mathrm{diag}(\eta_{n+1})L^{\top}$\;
}
Update Auxiliary Estimate:\\
$V_{1;n+1}^{(1)} = \mathrm{soft}([U_{n+1}^{(1)}+\mathcal{V}_{1;n}^{(1)}]_{\backslash\backslash ii};\:\lambda)$ \;
\For{$t=2,\ldots,T$}{
$P^{t}=[(U_{n+1}^{(t)}+\mathcal{V}_{1;n}^{(t)})+(V_{2}^{(t-1)}+W_{n}^{(t-1)}-\mathcal{W}_{n}^{(t-1)})]/2$\;
$V_{1;n+1}^{(t)} = \mathrm{soft}([P]_{\backslash\backslash ii};\:\lambda)$ \;
$V_{2;n+1}^{(t-1)}=(1/2)[(U_{n+1}^{(t-1)}+\mathcal{V}_{n}^{(t-1)}) + (V_{1;n}^{(t)}-W_{n}^{(t-1)} + \mathcal{W}_{n}^{(t-1)})]$\;
Threshold for jumps:\\
$Q^{(t)} = V_{1;n+1}^{(t)} - V_{2;n+1}^{(t-1)} + \mathcal{W}_{n}^{(t-1)}$\;
$W^{t-1}_{n+1} = ( Q^{(t)} / \|Q^{(t)}\|_F)\max (\|Q^{(t)}\|_F-\lambda\rho,0)$\;
}
Update Dual variables:\\
$\mathcal{V}_{1;n+1}^{(t)}=\mathcal{V}_{1;n}^{(t)}+U_{n+1}^{t}-V_{1;n}^{(t)}$ \;
$\mathcal{V}_{2;n+1}^{(t)}=\mathcal{V}_{2;n}^{(t)}+U_{n+1}^{t}-V_{2;n}^{(t-1)}$ \;
$\mathcal{W}_{n+1}^{(t)}=\mathcal{W}_{n}^{(t)}+V_{1;n+1}^{(t)}-V_{2;n}^{(t-1)}-W_{n+1}^{(t)}$ \;
}
\KwRet{$\{ \hat{\Theta}^{(t)} \}$}

\caption{ADMM algorithm for the minimisation of the GFGL cost function. \label{alg:ADMM_full}}
\end{algorithm}

\subsection{Proof of Proposition \ref{prop:optimality_conditions} (Optimality
Conditions)}

\label{subsec:proof_optimaity_GFGL}

In GFGL we have a set of conditions for each time-point which must
be met jointly. Unlike non-fused estimators, we also have to consider
the stationarity conditions due to a differenced term. The GFGL objective
can then be re-written in terms of this difference, where one may equivalently minimise
\[
\sum_{t=1}^{T}\bigg(-\log\det(\sum_{s\le t}\Gamma^{(s)})+\mathrm{tr}(\hat{S}^{(t)}\sum_{s\le t}\Gamma^{(s)})\bigg)+\lambda_{1}\sum_{t=1}^{T}\|\sum_{s\le t}\Gamma_{\backslash ii}^{(s)}\|_{1}+\lambda_{2}\sum_{t=2}^{T}\|\Gamma^{(t)}\|_{F}\;,
\]
for $\sum_{s\le t}\Gamma^{(s)} \succ 0 $ for $t=1,\ldots,T$.
Setting the derivative to zero we obtain:
\[
0=\sum_{t=l}^{T}\bigg(-(\sum_{s\le t}\hat{\Gamma}^{(s)})^{-1}+\hat{S}^{(t)}\bigg)+\lambda_{1}\sum_{t=l}^{T}\hat{R}_{1}^{(t)}+\lambda_{2}\hat{R}_{2}^{(l)}\;.
\]
The above derivative, is linked through the data via the function $\hat{S}^{(t)}=X^{(t)}(X^{(t)})^\top$.
Recalling that $X^{(t)}\sim\mathcal{N}_{p}(0,\Sigma^{(t)})$,
we can then link $\Sigma^{(t)}$ to $\Gamma^{(l)}$
via $\Sigma^{(t)}=(\sum_{s\le t}\Gamma^{(s)})^{-1}$.
We now have a way to relate the estimated precision matrices $\hat{\Gamma}^{(t)}$
and the corresponding ground-truth. Let us write the sampling error for the covariance matrix at time point $t$ as $\Psi^{(t)}:=\hat{S}^{(t)}-\Sigma^{(t)}$.
Substituting $\Psi^{(t)}$ into the above stationarity
conditions for GFGL we obtain;
\[
\sum_{t=l}^{T}\bigg((\sum_{s\le t}\Gamma^{(s)})^{-1}-(\sum_{s\le t}\hat{\Gamma}^{(s)})^{-1}\bigg)-\sum_{t=l}^{T}\Psi^{(t)}+\lambda_{1}\sum_{t=l}^{T}\hat{R}_{1}^{(t)}+\lambda_{2}\hat{R}_{2}^{(l)}\;,
\]
and thus equivalently obtaining the result in Proposition \ref{prop:optimality_conditions}.\hfill\ensuremath{\square}


\section{Proof of changepoint consistency}

We relate the proof bounding the maximum deviation between estimated
and true changepoints to the probability of an individual changepoint
breaking the bound. Following \citet{Harchaoui2010},
we utilise the union bound
\[
P[\max_{k\in[K]}|\tau_{k}-\hat{\tau}_{k}|\ge T\delta_{T}]\le\sum_{k\in[K]}P[|\tau_{k}-\hat{\tau}_{k}|\ge T\delta_{T}]\;.
\]
The compliment of the event on the LHS is equivalent to the
target of proof; we wish to demonstrate $P[\max_{k\in[K]}|\tau_{k}-\hat{\tau}_{k}|\le T\delta_{T}]\rightarrow1$.
In order to show this, we need to show the LHS above goes to zero
as $T\rightarrow\infty$. It is sufficient, via the union bound, to
demonstrate that the probability of the "bad" events:
\begin{equation}
\label{eq:bad_event}
A_{T,k}:=\{|\tau_{k}-\hat{\tau}_{k}|>T\delta_{T}\}\;,
\end{equation}
go to zero for all $k\in[K]$. The strategy presented here separates
the probability of $A_{T,k}$ occurring across complimentary events.
In particular, let us construct what can be thought of as a good event,
where the estimated changepoints are within a region of the true ones:

\begin{equation}
\label{eq:good_event}
C_{T}:=\left\{ \underset{k\in[K]}{\max}|\hat{\tau}_{k}-\tau_{k}|<\frac{d_{\min}}{2}\right\} \;.
\end{equation}
The task is then to show that $P[A_{T,k}]\rightarrow0$ by showing
$P[A_{T,k}\cap C_{T}]\rightarrow0$ and $P[A_{T,k}\cap C_{T}^{c}]\rightarrow0$
as $T\rightarrow0$.

\subsection{Stationarity induced bounds}

As a first step, let us introduce some bounds based on the optimality conditions which occur in
probability one. From here, a
set of events can be constructed that occur when the stationarity
conditions are met. By intersecting these events with
$A_{T,k}\cap C_{T}$ and $A_{T,k}\cap C_{T}^{c}$, we can construct an upper bound on the probability for changepoint error exceeding a level $T\delta_T$.

Without loss of generality, consider the stationarity equations (Prop.
\ref{prop:optimality_conditions}) with changepoints $l=\tau_{k}$
and $l=\hat{\tau}_{k}$ such that $\hat{\tau}_{k}<\tau_{k}$. We note, that an argument for the reverse situation $\tau_{k}>\hat{\tau}_{k}$ follows through symmetry. Taking the differences between the equations
we find

\begin{equation}
\|\sum_{t=\hat{\tau}_{k}}^{\tau_{k}-1}(\Sigma^{(t)}-\hat{\Sigma}^{(t)})-\sum_{t=\hat{\tau}_{k}}^{\tau_{k}-1}\Psi^{(t)}+\lambda_{1}\sum_{t=\hat{\tau}}^{\tau_{k}-1}\hat{R}_{1}^{(t)}\|_{F}\le2\lambda_{2}\;.\label{eq:stationarity_difference_1}
\end{equation}
The gradient from the $\ell_{1}$ term $\sum_{t=\hat{\tau}_{k}}^{\tau_{k}-1}\lambda\hat{R}_{1}^{(t)}$
can obtain a maximum value of $\pm\lambda_{1}(\tau_{k}-\hat{\tau})$
for each entry in the precision matrix. Transferring this to the RHS and splitting the LHS in terms of the stochastic and estimated terms we obtain

\begin{equation}
\|\sum_{t=\hat{\tau}_{k}}^{\tau_{k}-1}(\Sigma_0^{(t)}-\hat{\Sigma}^{(t)})\|_{F}-\|\sum_{t=\hat{\tau}_{k}}^{\tau_{k}-1}\Psi^{(t)}\|_{F}\le2\lambda_{2}+\lambda_{1}\sqrt{p(p-1)}(\tau_{k}-\hat{\tau}_{k})\;.\label{eq:stationarity_diff}
\end{equation}
The next step is to replace the time indexed inverse precision matrices
$\Theta^{(t)}$ with the block-covariance matrices indexed
$\Sigma^{(k)}$ and $\Sigma^{(k+1)}$. We
can re-express the difference in precision matrices as the sum of
a difference between true values before $\tau_{k}$ , i.e. $\Sigma^{(k+1)}-\Sigma^{(k)}$,
and the difference between the next ($k+1$)st true block and estimated
block, i.e. $\hat{\Sigma}^{(k+1)}-\Sigma^{(k+1)}$
to obtain:

\begin{eqnarray}
\lambda_{2}+\lambda_{1}\sqrt{p(p-1)}(\tau_{k}-\hat{\tau}_{k}) & \ge & \underset{\|R_{1}\|_{F}}{\underbrace{\|\sum_{t=\hat{\tau}_{k}}^{\tau_{k}-1}\Sigma^{(k)}-\Sigma^{(k+1)}\|_{F}}}-\underset{\|R_{2}\|_{F}}{\underbrace{\|\sum_{t=\hat{\tau}_{k}}^{\tau_{k}-1}\hat{\Sigma}^{(k+1)}-\Sigma^{(k+1)}\|_{F}}}\nonumber \\
 &  & \quad -\underset{\|R_{3}\|_{F}}{\underbrace{\|\sum_{t=\hat{\tau}_{k}}^{\tau_{k}-1}\Psi^{(t)}\|_{F}}}\;,\label{eq:stationarity_seperated}
\end{eqnarray}
which holds with probability one. Define the events:

\begin{align*}
E_{1}:= & \{\lambda_{2}+\lambda_{1}\sqrt{p(p-1)}(\tau_{k}-\hat{\tau}_{k})\ge\frac{1}{3}\|R_{1}\|_{F}\}\\
E_{2}:= & \{\|R_{2}\|_{F}\ge\frac{1}{3}\|R_{1}\|_{F}\}\\
E_{3}:= & \{\|R_{3}\|_{F}\ge\frac{1}{3}\|R_{1}\|_{F}\}
\end{align*}
Since we know that the bound (\ref{eq:stationarity_seperated}) occurs
with probability one, then the union of these three events must also
occur with probability one, i.e. $P[E_{1}\cup E_{2}\cup E_{3}]=1$.

\subsection{Bounding the Good Cases}

One of the three events above are required to happen, either together,
or separately. We can thus use this to bound the probability of both
the good $C_{T}$ and bad $A_{T,k}$ events. Similarly to \citet{Harchaoui2010,Kolar2012}
we obtain
\begin{eqnarray*}
P[A_{T,k}\cap C_{T}] & \le & P[\overset{A_{T,k,1}}{\overbrace{A_{T,k}\cap C_{T}\cap E_{1}}}]+P[\overset{A_{T,k,2}}{\overbrace{A_{T,k}\cap C_{T}\cap E_{2}}}]+P[\overset{A_{T,k,3}}{\overbrace{A_{T,k}\cap C_{T}\cap E_{3}}}]
\end{eqnarray*}
The following sub-sections describe how to separately bound these
sub-events.

Unlike in the work of \citet{Kolar2012}, there is no stochastic element
(related to the data $X_{t}$) within the first event $A_{T,k,1}$.
We can bound the probability of $P[A_{T,k,1}]$ by considering the
event $\{\frac{1}{3}\|R_{1}\|_{F}\le\lambda_{2}+\lambda_{1}\sqrt{p(p-1)}(\tau_{k}-\hat{\tau}_{k})\}$.
Given $\|R_{1}\|_{F}=\|\sum_{t=\hat{\tau}_{k}}^{\tau_{k}-1}\Sigma^{(k)}-\Sigma^{(k+1)}\|_{F}\ge(\tau_{k}-\hat{\tau}_{k})\eta_{\min}$
we therefore obtain the bound
\[
P[A_{T,k,1}]\le P[(\tau_{k}-\hat{\tau}_{k})\eta_{\min}/3\le\lambda_{2}+\lambda_{1}\sqrt{p(p-1)}(\tau_{k}-\hat{\tau}_{k})]\;.
\]
When the events $C_{T},A_{T,k}$ occur we have $T\delta_{T}<\tau_{k}-\hat{\tau}_{k}\le d_{\min}/2$
to ensure the event $A_{T,k,1}$ does not occur, we need:
\begin{equation}
\eta_{\min}T\delta_{T}>3\lambda_{2}\quad;\quad \eta_{\min}>3\lambda_{1}\sqrt{p(p-1)}\;.\label{eq:A_T_k_1_convergence_criteria}
\end{equation}
These conditions are satisfied by Assumption \ref{ass:appropriate_reg}. Thus, for a large enough $T$, we can show that
the probability $P[A_{T,k,1}]=0$, the size of this $T$ depends on
the quantities in Eq. \ref{eq:A_T_k_1_convergence_criteria}.

Now let us consider the event $A_{T,k,2}$. Consider the quantity
$\bar{\tau}_{k}:=\lfloor(\tau_{k}+\tau_{k+1})/2\rfloor$. On the event
$C_{n}$, we have $\hat{\tau}_{k+1}>\bar{\tau}_{k}$ so $\hat{\Sigma}^{(t)}=\hat{\Sigma}^{(k+1)}$
for all $t\in[\tau_{k},\bar{\tau}_{k}]$. Using the optimality conditions
(Prop \ref{prop:optimality_conditions}) with changepoints at $l=\bar{\tau}_{k}$
and $l=\tau_{k}$ we obtain
\[
2\lambda_{2}+\lambda_{1}\sqrt{p(p-1)}(\bar{\tau}_{k}-\tau_{k}) \ge \|\sum_{t=\tau_{k}}^{\bar{\tau}_{k}-1}\hat{\Sigma}^{(k+1)}-\Sigma^{(k+1)}\|_{F}-\|\sum_{t=\tau_{k}}^{\bar{\tau}_{k}-1}\Psi^{(t)}\|_{F}\;,
\]
and thus
\begin{equation}
\|\hat{\Sigma}^{(k+1)}-\Sigma^{(k+1)}\|_{F} \le \frac{4\lambda_{2}+2\lambda_{1}\sqrt{p(p-1)}(\bar{\tau}_{k}-\tau_{k})+2\|\sum_{t=\tau_{k}}^{\bar{\tau}_{k}-1}\Psi^{(t)}\|_{F}}{\tau_{k+1}-\tau_{k}}\label{eq:that-t}
\end{equation}
We now combine the bounds for events $E_{1}$ and $E_{2}$, via $E_{2}:=\{\|R_{2}\|_{F}\ge\frac{1}{3}\|R_{1}\|_{F}\}$
and the bounds $\|R_{1}\|_{F}\ge(\tau_{k}-\hat{\tau}_{k})\eta_{\min}$
and $\|R_{2}\|_{F}\le(\tau_{k}-\hat{\tau}_{k})\|\hat{\Sigma}^{k+1}-\Sigma^{k+1}\|_{F}$
. Substituting in (\ref{eq:that-t}) we have

\begin{equation}
P[A_{T,k,2}]\le P[E_{2}]=P\left[\eta_{\min}\le\frac{12\lambda_{2}+6\lambda_{1}\sqrt{p(p-1)}(\bar{\tau}_{k}-\tau_{k})+6\|\sum_{t=\tau_{k}}^{\bar{\tau}_{k}-1}\Psi^{(t)}\|_{F}}{\tau_{k+1}-\tau_{k}}\right]\;.\label{eq:concentration_bound_ATk2}
\end{equation}
Splitting the probability into three components, we obtain
\begin{equation}
P[A_{T,k,2}]\le P[\eta_{\min}d_{\min}\le12\lambda_{2}]+P[\eta_{\min}\le3\lambda_{1}\sqrt{p(p-1)}]+P\left[\eta_{\min}\le\frac{6\|\sum_{t=\tau_{k}}^{\bar{\tau}_{k}-1}\Psi^{(t)}\|_{F}}{\tau_{k+1}-\tau_{k}}\right]\;.\label{eq:sum_of_probs_ATK2_sd}
\end{equation}
Convergence of the first two terms follows as in $A_{T,k,1}$, the
second is exactly covered in $A_{T,k,1}$; however, the third term
$\eta_{\min}\le3\|\sum_{t=\tau_{k}}^{\bar{\tau}_{k}-1}\Psi^{(t)}\|_{F}/(\bar{\tau}_{k}-\tau_{k})$
requires some extra treatment. As $\bar{\tau}_{k}<\tau_{k+1}$, we
can relate the covariance matrix of the ground-truth (time-indexed)
and block (indexed by $k$) such that $\Sigma^{(t)}=\Sigma^{(k)}$
for all $t\in[\tau_{k},\tau_{k+1}]$. One can now write the sampling
error across time into one which relates to blocks $k$ as

\[
\|\sum_{t=\tau_{k}}^{\bar{\tau}_{k}-1}\Psi^{(t)}\|_{F}\equiv(\bar{\tau}_{k}-\tau_{k})\|W_{k;\bar{\tau}_{k}-\tau_{k}}\|_{F}\;,
\]
where $W_{k;n}=[n^{-1}\sum_{t=1}^n X^{(t)}(X^{(t)})^{\top} ] - \Sigma_{0}^{(k)}$. A control on this quantity is given in the following lemma:

\begin{lemma}{Sample Error Bound in High-Dimensions}

\label{lemma:tail_bound_hd}
Let $\hat{W}_k^{(n)}=[n^{-1}\sum_{t=1}^n X^{(t)}(X^{(t)})^{\top} ] - \Sigma_{0}^{(k)}$, then for any $\epsilon\in(0,2^3c_{\sigma}p)$, with sub-Gaussian noise $\{X^{(t)} \}_{t=1}^T$, the error is bounded according to
\[
P(\|\hat{W}_{k;n}\|_{F}>\epsilon)\le4p^{2}\exp\left(-\frac{n\epsilon^{2}}{2^{7}c_{\sigma}^{2}p^{2}}\right)\;,
\]
where $c_\sigma = (1+4\sigma^2)\max_{ii}\{\Sigma^{(k)}_{0;ii}\}$.
Furthermore, if $\epsilon>
\epsilon_{\mathrm{conv}}^{\alpha}:=c_\sigma 2^{4}\sqrt{2}p\sqrt{\log(p^{\beta/2})/n}
$ and $\beta>2$ then 
$P(\|\hat{W}_{k;n}\|_{F}>\epsilon)\le p^{(2-\beta)} \rightarrow 0$ as $p\rightarrow \infty$.

\begin{proof}
See section \ref{proof:empirical_covariance_frobenius_bound}.
\end{proof}
\end{lemma}

In the specific setting of $A_{T,k,2}$ the probability we desire to bound is $P[\|W_{k,\bar{\tau}_{k}-\tau_{k}}\|_{F}>\eta_{\min} /3]$, applying Lemma \ref{lemma:tail_bound_hd} gives convergence if $\eta_{\mathrm{min}}>15\times 2^5 p\sqrt{\log (p^{\beta/2})/d_{\min} }$, where we note $\bar{\tau}_{k}-\tau_{k} >d_{\mathrm{min}}/2$ and for Gaussian sampling $c_\sigma=5$. Finally, to ensure that the bound in Lemma \ref{lemma:tail_bound_hd} holds, we must ensure that $\eta_{\min}/3\in (0,2^3c_{\sigma}/p)$. If $d_{\min}=\gamma_{\min}T$ and $\eta_{\min}$ follows the form $c p \sqrt{\log (p^{\beta/2}/\gamma_{\min}T)}$ for some constant $c$, then we require $T> \gamma_{\min}^{-1}(c/c_\sigma)^2 2^{-6} \log(p^{\beta/2})$.

Alternatively, we may study the sampling error in the standard asymptotic setting where $\bar{\tau}_k-\tau_k>p$.

\begin{lemma}{Sample Error Bound in Standard Dimensions}
\label{lemma:tail_bound_sd}

Let $p\le n$ and $X^{(t)}\sim\mathcal{N}(0,\Sigma_0^{(k)})$
for $t=1,\ldots,n$. If the covariance matrix $\Sigma_0^{(k)}$ has maximum eigenvalues
$\phi_{\max}<+\infty$, then for all $t>0$ and $p>2$ we have
\begin{equation}
\label{eq:loose_sd_1}
P\left(\|\hat{W}_{k;n}\|_F \ge 4\phi_{\max}e^{-1/2}\sqrt{p\log n/n} \right) < 2n^{-p/2}\;.
\end{equation}

\begin{proof}
The result is a corollary of bounds derived in \citep{Wainwright2009}, see \ref{proof:tail_bound_sd}.
\end{proof}
\end{lemma}

In the context of $A_{T,k,2}$, Lemma \ref{lemma:tail_bound_sd} holds with $\eta_{\min}\ge 12\phi_{\max}e^{-1/2}\sqrt{p \log n/n}$ and $n=\gamma_{\min}T$. While the probability bound in the Lemma is not sharp, it does relate the convergence rate to the dimensionality $p$, which in this context is fixed as $T\rightarrow \infty$. 

Finally, let us turn to $A_{T,k,3}$. Recall $P(A_{T,k,3}):=P(A_{T,k}\cap C_{T}\cap E_{3}):=P(A_{T,k}\cap C_{T}\cap\{\|\sum_{t=\hat{\tau}_{k}}^{\tau_{k}-1}\Psi^{(t)}\|_{F}\ge \|R_{1}\|_{F}/3\})$.
Given that $\|R_{1}\|_{F}\ge(\tau_{k}-\hat{\tau}_{k})\eta_{\min}$
with probability 1, an upper bound on $P[A_{T,k,3}]$ can be found
using the same concentration bounds (Lemmas \ref{lemma:tail_bound_hd}, \ref{lemma:tail_bound_sd}) as for $A_{T,k,2}$. The only difference is that we need to replace the integration interval $n$ with $T\delta_T$. Noting that $T\delta_{T}<\tau_{k}-\hat{\tau}_{k}\le d_{\min}/2$, the overall bound will be dominated by the concentration results requiring $n>T\delta_T$.

\subsection{Bounding the Bad Cases}

In order to complete the proof, we need to demonstrate that $P[A_{T,k}\cap C_{T}^{c}]\rightarrow0$.
The argument below follows that of \citet{Harchaoui2010}, whereby
the bad case is split into several events:
\begin{eqnarray*}
D_{T}^{(l)}: & = & \left\{ \exists k\in[K],\;\hat{\tau}_{k}\le\tau_{k-1}\right\} \cap C_{T}^{c},\\
D_{T}^{(m)}: & = & \left\{ \forall k\in[K],\;\tau_{k-1}<\hat{\tau}_{k}<\tau_{k+1}\right\} \cap C_{T}^{c},\\
D_{T}^{(r)}: & = & \left\{ \exists k\in[K],\;\hat{\tau}_{k}\ge\tau_{k+1}\right\} \cap C_{T}^{c},
\end{eqnarray*}
where $C_{T}^{c}=\{\max_{k\in[K]}|\hat{\tau}_{k}-\tau_{k}|\ge d_{\min}/2\}$
is the compliment of the good event. The events above correspond to
estimating a changepoint; a) before the previous true changepoint
($D_{T}^{(l)}$); b) between the previous and next true changepoint
($D_{T}^{(m)}$), and c) after the next true changepoint ($D_{T}^{(r)}$).
The events $D_{T}^{(l)}$ and $D_{T}^{(r)}$ appear to be particularly
bad as the estimated changepoint is very far from the truth, due to
symmetry we can bound these events in a similar manner. Focussing
on the middle term $P[A_{T,k}\cap D_{T}^{(m)}]$, let us again assume
$\hat{\tau}_{k}<\tau_{k}$ , the reverse arguments hold by symmetry. 

\begin{lemma}{Upper bound for $P[A_{T,k}\cap D_{T}^{(m)}]$ }
\label{lemma:upper_bound_a_tk_cap_D_tm}

The probability of the intersection of $A_{T,k}$ and $D_{T}^{(m)}$
can be bounded from above by considering the events 
\begin{eqnarray}
E_{k}^{'} & := & \{(\hat{\tau}_{k+1}-\tau_{k})\ge d_{\min}/2\}\;,\\
E_{k}^{''} & := & \{(\tau_{k}-\hat{\tau}_{k})\ge d_{\min}/2\}\;.\label{eq:Events_compliment_bound_k+1}
\end{eqnarray}
In particular, one can demonstrate that:
\begin{equation}
P[A_{T,k}\cap D_{T}^{(m)}]\le P[A_{T,k}\cap E_{k}^{'}\cap D_{T}^{(m)}]+\sum_{j=k+1}^{K}P[E_{j}^{''}\cap E_{j}^{'}\cap D_{T}^{(m)}]\;.\label{eq:compliment_middle_split}
\end{equation}
\end{lemma}
\begin{proof}
The result follows from expanding events based on neighbouring changepoints
(see Appendix \ref{subsec:upper_bound_DtM} for detail). \end{proof}

Let us first assess $P(A_{T,k}\cap D_{T}^{(m)}\cap E_{k}^{'})$, and
consider the stationarity conditions (\ref{eq:stationarity_diff})
with start and end points set as $l=\hat{\tau}$, $l=\tau_{k}$ and $l=\hat{\tau}_{k},l=\tau_{k+1}$
. We respectively obtain:
\begin{equation}
\label{eq:bound_cp_left_side_k+1_k}
|\tau_{k}-\hat{\tau}_{k}|\|\Sigma^{(k)}-\hat{\Sigma}^{(k+1)}\|_{F}\le2\lambda_{2}+\lambda_{1}\sqrt{p(p-1)}(\tau_{k}-\hat{\tau}_{k})+\|\sum_{t=\hat{\tau}_{k}}^{\tau_{k}-1}\Psi^{(t)}\|_{F}\;
\end{equation}
and
\begin{equation}
\label{eq:bound_cp_left_side}
|\tau_{k}-\hat{\tau}_{k+1}|\|\Sigma^{(k+1)}-\hat{\Sigma}^{(k+1)}\|_{F}\le2\lambda_{2}+\lambda_{1}\sqrt{p(p-1)}(\hat{\tau}_{k+1}-\tau_{k})+\|\sum_{t=\tau_{k}}^{\hat{\tau}_{k+1}-1}\Psi^{(t)}\|_{F}\;.
\end{equation}
The next step is to define an event that can bound $P(A_{T,k}\cap E_{k}^{'}\cap D_{T}^{(m)})$.
Using the triangle inequality we bound $\|\Sigma^{(k+1)}-\Sigma^{(k)}\|_{F}$
conditional on $E_{k}^{'}:=\{(\hat{\tau}_{k+1}-\tau_{k})\ge d_{\min}/2 \}$
and $A_{T,k}:=\{|\tau_{k}-\hat{\tau}_{k}|>T\delta_{T}\}$. Specifically, we construct the event
\begin{eqnarray}
H_{T}^{\Sigma} & := & \{ \|\Sigma_{k+1}-\Sigma_{k}\|_{F}\le 2\lambda_{1}\sqrt{p(p-1)}+2\lambda_{2}( (T\delta_{T})^{-1}+2/d_{\min} ) \nonumber \\
 &  & \quad +\|W_{k;\tau_{k}-\hat{\tau}_{k}}\|_F + \|W_{k+1;\hat{\tau}_{k+1}-\tau_{k}}\|_F \}\;, \label{eq:H_sigma}
\end{eqnarray}
which bounds the first term of (\ref{eq:compliment_middle_split})
such that $P(A_{T,k}\cap E_{k}^{'}\cap D_{T}^{(m)}) \le P(H_{T}^{\Sigma}\cap\{\tau_{k}-\hat{\tau}_{k}\ge T\delta_{T}\}\cap E_{k}^{'})$.
Splitting the intersection of events we now have five terms to consider
\begin{eqnarray*}
 &  & P(A_{T,k}\cap E_{k}^{'}\cap D_{T}^{(m)})\\
 & \le & P(\lambda_{1}\sqrt{p(p-1)}\ge \eta_{\min}/10 ) + P(\lambda_{2}/T\delta_{T} \ge \eta_{\min}/10 ) + P(\lambda_{2}/d_{\min} \ge  \eta_{\min}/20) \\
 &  & +P( \|W_{k;\tau_{k}-\hat{\tau}_{k}} \|_F \ge \eta_{\min}/5 \} \cap \{ \tau_{k}-\hat{\tau}_{k} \ge T\delta_{T} \} )\\
 &  & +P( \{ \|W_{k+1;\hat{\tau}_{k+1}-\tau_{k}}\|_F \ge \eta_{\min}/5  \} \cap  \{ \hat{\tau}_{k+1}-\tau_{k}\ge d_{\min}/2 \} )\;.
\end{eqnarray*}
The stochastic error terms (containing $W_{k;\tau_k-\hat{tau}_k}$)
can then be shown to converge similarly to $P(A_{T,k}\cap C_{T})$
c.f. Eq. (\ref{eq:concentration_bound_ATk2}). Again, it is worth
noting that the term involving $T\delta_{T}$ will be slowest to converge,
as $d_{\min}=\gamma_{\min}T>\delta_{T}T$ for large $T$. The first
three terms are bounded through the assumptions on $d_{\min},\lambda_{1},\lambda_{2}$,
and $\delta_{T}$ as required by the theorem (and enforce a similar
requirement to those used to bound $P(A_{T,k,1})$ in Eq. \ref{eq:A_T_k_1_convergence_criteria}).
The other terms in (\ref{eq:compliment_middle_split}), i.e. $\sum_{j=k+1}^{K}P[E_{j}^{''}\cap E_{j}^{'}\cap D_{T}^{(m)}]$
can be similarly bounded. Instead of using exactly the event $H_{T}^{\Sigma}$
one simply replaces the term $1/T\delta_{T}$ in (\ref{eq:H_sigma})
with $2/d_{\min}$.

Now let us consider the events $D_{T}^{(l)}:=\left\{ \exists k\in[K],\;\hat{\tau}_{k}\le\tau_{k-1}\right\} \cap C_{T}^{c}$.
The final step of the proof is to show that the bound on $A_{T,k}\cap D_{T}^{(l)}$,
and similarly $A_{T,k}\cap D_{T}^{(r)}$ tends to zero:

\begin{lemma}The probability of $D_{T}^{(l)}$ is bounded by 
\[
P(D_{T}^{(l)}) \le 2^{K}\sum_{k=1}^{K-1}\sum_{l\ge k}^{K-1}P(E_{l}^{''}\cap E_{l}^{'})+2^{K}P(E_{K}^{'})\;.
\]
\label{lemma:union_bound_left}
\begin{proof}
This is based on a combinatorial argument for the events that can
be considered on addition of each estimated changepoint. For details
see Appendix \ref{subsec:proof_lemma_union_bound_left}.\end{proof}
\end{lemma}

In order to bound the above probabilities we relate the events $E_{l}^{''}$and
$E_{l}^{'}$ to the stationarity conditions as before (via Eqs. \ref{eq:bound_cp_left_side_k+1_k},
\ref{eq:bound_cp_left_side}). Setting $k=l$ and invoking the triangle
inequality gives us
\begin{eqnarray*}
\{\|\Sigma_{l+1}-\Sigma_{l}\|_{F} & \le & 2\lambda_{1}\sqrt{p(p-1)}+\overset{A}{\overbrace{2\lambda_{2}(|\tau_{l}-\hat{\tau}_{l}|^{-1}+|\hat{\tau}_{l+1}-\tau_{l}|^{-1}) }} \\
 &  & + \| W_{l;\tau_{l}-\hat{\tau}_{l}}\|_F+ \|W_{l+1;\hat{\tau}_{l+1}-\tau_{l}}\|_F \}\;.
\end{eqnarray*}
Conditioning on the event $E_{l}^{''}\cap E_{l}^{'}$ implies that
$A=8\lambda_{2}/d_{\min}$. We can thus write
\begin{eqnarray*}
P(E_{l}^{''}\cap E_{l}^{'}) & \le & P(\eta_{\min} \le 8\lambda_{1}\sqrt{p(p-1)}) +P(\eta_{\min}\le 32 \lambda_{2}/d_{\min} )\\
 &  & +P(\{ \|W_{l;\tau_{l}-\hat{\tau}_{l}}\|_F \ge \eta_{\min}/4 \}\cap \{ \tau_{l}-\hat{\tau}_{l}\ge d_{\min}/2 \} )\\
 &  & +P ( \{ \|W_{l+1;\hat{\tau}_{l+1}-\tau_{l}} \|_F \ge \eta_{\min}/4\} \cap \{ \hat{\tau}_{l+1}-\tau_{l} \ge d_{\min}/2 \} )\;.
\end{eqnarray*}
Finally, the term corresponding to the last changepoint can be bounded
by noting that when $k=K$ we have $A=6\lambda_{2}/d_{\min}$.
\begin{align}
P(E_{K}^{''}) & \le P(\eta_{\min}\le 8\lambda_{1}\sqrt{p(p-1)} )+P(\eta_{\min}\le 24\lambda_{2}/d_{\min} ) \nonumber \\
 & +P( \{ \|W_{K;\tau_{K}-\hat{\tau}_{K}}\|_F \ge \eta_{\min}/4 \} \cap \{\tau_{K}-\hat{\tau}_{K}\ge d_{\min}/2 \}) \nonumber \\
 & P( \|W_{K+1;T+1-\tau_{K}}\|_F \ge \eta_{\min}/4 )\;.\label{eq:E_K''}
\end{align}

\subsection{Summary}

The bounds derived above demonstrate that $P(A_{T,k}) \rightarrow 0$
since $P(A_{T,k}\cap C_{T})\rightarrow 0$ and $P(A_{T,k}\cap C_{T}^{c})\rightarrow0$. However, to achieve these bounds, the regularisers must be set appropriately. The event $E_{l}^{''}\cap E_{l}^{'}$ establishes
a minimal condition on $T$ in conjunction with $\eta_{\min}$ and the regularisers, such that $\eta_{\min}d_{\min}/\lambda_{2}>32$
and $\eta_{\min}/\lambda_{1}\sqrt{p(p-1)}>8$. A final condition for
$A_{T,k,1}$ requires $\eta_{\min}T\delta_{T}/\lambda_{2}>3$. Once
$T$ is large enough to satisfy these conditions, the probabilistic
bound is determined either by the smallest block size $d_{\min}=\gamma_{\min}T$
or by the minimum error $T\delta_{T}$. Let $k_\infty = \arg\max_k\{\max_{ii}\Sigma_{ii}^{(k)} \}$ select the block which results in the largest expected covariance error. Summing the probabilities,
one obtains the upper bound:
\begin{align*}
P[|\tau_{k}-\hat{\tau}_{k}|\ge T\delta_{T}]\le & 2\times2^{K}\big((K-1)^{2}+1\big)P(\|W_{k_\infty;d_{\min}/2}\|_F \ge \eta_{\min}/4 )\\
 & + 2 P(\|W_{k_\infty;T\delta_T}\|_F \ge \eta_{\min}/5) \\
 & +2 P(\|W_{k_\infty;T\delta_T}\|_F \ge \eta_{\min}/3) \;,
\end{align*}
where the different rows correspond to events; top) $D_{T}^{(l)}$
and $D_{T}^{(r)}$; middle) $D_{T}^{(m)}$; bottom) $A_{T,k,2}$
and $A_{T,k,3}$. Since $\delta_{T}T<\gamma_{\min}T$ the above bounds
will be dominated by errors $W_{k_\infty;T\delta_T}$ integrated over the relatively small distance $T\delta_T$. A suitable overall bound on
the probability is
\begin{align*}
P(\max_{k\in[K]}|\tau_{k}-\hat{\tau}_{k}|\ge T\delta_{T}) & \le K^{3}2^{K+1} P(\|W_{\infty;d_{\min}/2}\|_F \ge \eta_{\min}/4) \\
& \quad + 4K P(\|W_{\infty;T\delta_T}\|_F \ge \eta_{\min}/5)\\
 & \le C_{K}P(\|W_{\infty;T\delta_T}\|_F \ge \eta_{\min}/5)\;,
\end{align*}
where $C_{K}=K(K^{2}2^{K+1}+4)$. We thus arrive at the result of
Theorem \ref{thm:cp_consistency_sd}.

\hfill\ensuremath{\square}

\subsection{Lemma 1. High-Dimensional Bound on Empirical Covariance Error}

\label{proof:empirical_covariance_frobenius_bound} 

\begin{lemma*}

Let $W_k^{(n)}=[n^{-1}\sum_{t=1}^n X^{(t)}(X^{(t)})^{\top} ] - \Sigma_{0}^{(k)}$, then for any $\epsilon\in(0,2^3c_{\sigma}/p)$, with Gaussian noise $X^{(t)}\sim \mathcal N(0,\Sigma^{(k)})$, the error is bounded according to
\[
P(\|\hat{W}_{k;n}\|_{F}>\epsilon)\le4p^{2}\exp\left(-\frac{n\epsilon^{2}}{2^{7}c_{\sigma}^{2}p^{2}}\right)\;,
\]
where $c_\sigma = (1+4\sigma^2)\max_{ii}\{\Sigma^{(k)}_{0;ii}\}$.
Furthermore, if $\epsilon>
\epsilon_{\mathrm{conv}}^{\alpha}:=c_\sigma 2^{4}\sqrt{2}p\sqrt{\log(p^{\alpha/2})/n}
$ and $\alpha>2$ then 

\[
P(\|\hat{W}_{k;n}\|_{F}>\epsilon)\le p^{(2-\alpha)} \rightarrow 0\;.
\]

\begin{proof}

From \citet{Ravikumar2011} Lemma 1, for $X^{(t)}$ with sub-Gaussian tails (parameter
$\sigma$) we have

\[
P(|\hat{W}_{k;n;i,j}|>\delta)\le4\exp\left(-\frac{n\delta^{2}}{2^{7}c^{2}}\right)\;,
\]
where $c=(1+4\sigma^{2})\max_{ii}\{\Sigma_{0;ii}\}$ for all $\delta\in(0,2^{3}c)$.
Take the union bound to obtain max norm
\[
P(\|\hat{W}_{k;n}\|_{\infty}>\delta)\le4p^{2}\exp\left(-\frac{n\delta^{2}}{2^{7}c^{2}}\right)\;.
\]
Now use the fact that $p^{2}\|X\|_{\infty}\ge p\|X\|_{F}$ to control
the event relating to Frobenius norm. Note that if $P(\|X\|_{\infty}>\delta/p)=A$
and $P(\|X\|_{F}>\delta)=B$, then $A\ge B$ and hence we can use the larger probability $A$
to bound the Frobenius event
\[
P(\|\hat{W}_{k;n}\|_{F})>\delta)\le4p^{2}\exp\left(-\frac{n\delta^{2}}{p^2 2^{7}c^{2}}\right)
\]
The bound converges to zero for all $\epsilon >\epsilon_{\mathrm{conv}}$
where 
\begin{align*}
\epsilon_{\mathrm{conv}}^{2} & =\frac{p^{2}2^{2}2^{7}c^{2}}{n}\log(4p^{2})
\end{align*}
giving
$
\epsilon_{\mathrm{conv}}:=c2^{4}\sqrt{2}p\sqrt{\log(p)/n}\;.
$

\end{proof}
\end{lemma*}

\subsection{Proof of Lemma . Standard-dimensional Bounds for Empirical Covariance Error}

\label{proof:tail_bound_sd}

\begin{lemma*}{Concentration of spectral and Frobenius norm}

Let $p\le n$ and $X^{(t)}\sim\mathcal{N}(0,\Sigma_0^{(k)})$
for $t=1,\ldots,n$. If the covariance matrix $\Sigma$ has maximum eigenvalues
$\phi_{\max}<+\infty$, then for all $a>0$
\begin{equation}
\label{eq:wainwright_concentration}
P(\vertiii{\hat{W}_{k;n}}_{2}\ge \phi_{\mathrm{max}}\delta(n,p,a))\le 2\exp(-n a^{2}/2)\;,
\end{equation}
where $\delta(n,p,a):=2((p/n)^{1/2}+a)+((p/n)^{1/2}+a)^{2}$. Furthermore, with $\delta(n,p,\sqrt{p \log(n)/n})$ and $p>2$ we have

\begin{equation}
\label{eq:loose_sd}
P\left(\|\hat{W}_{k;n}\|_F \ge \frac{4\phi_{\max}}{\sqrt{e}}\sqrt{\frac{p\log n}{n}} \right) < 2n^{-p/2}\;.
\end{equation}

\begin{proof}
The proof of Eq. \ref{eq:wainwright_concentration} is given in Lemma 9 \citep{Wainwright2009} and is based on a result for Gaussian ensembles from \citet{Davidson2001}.
For the specific bound in (\ref{eq:loose_sd}) set $a=\sqrt{p \log (n)/n}$ such that
\begin{align*}
\delta(n,p,\sqrt{p \log (n)/n}) & = \left(\sqrt{\frac{p}{n}}+\sqrt{\frac{p \log (n)}{n}} \right)\left( 2+\sqrt{\frac{p}{n}} + \sqrt{\frac{p \log (n)}{n}} \right) \; \\
& <\sqrt{\frac{p \log(n)}{n}}(2+2\sqrt{\frac{p\log (n)}{n}})\;.
\end{align*}
Noting that $\sqrt{p \log (n)/n}$ is maximised for $n=e$, we obtain
\begin{align*}
\delta(n,p,\sqrt{p \log (n)/n})& <\sqrt{\frac{p \log n}{n}}(2(1+\sqrt{\frac{p}{e}})) \\
& <\frac{4p}{\sqrt{e}}\sqrt{\frac{\log n}{n}}\;,
\end{align*}
where the last inequality holds for $p>2$. Note that $\|X\|_F\le \sqrt{\mathrm{rank}(X)}\vertiii{X}_2$. Given we are in the setting $p<n$ so the matrix is full rank, we thus obtain the stated result (\ref{eq:wainwright_concentration})
\[
P(\|\hat{W}_{k;n}\|_F \ge \phi_{\mathrm{max}}\delta(n,p,\sqrt{p \log n /n})/\sqrt{p}) \le 2n^{-p/2}\;.
\]

\end{proof}

\end{lemma*}

\subsection{Proof of Lemma \ref{lemma:upper_bound_a_tk_cap_D_tm}}

\label{subsec:upper_bound_DtM}
\begin{lemma*}The probability of the intersection of $A_{T,k}$ and
$D_{T}^{(m)}$ can be bounded from above by considering the events
\begin{eqnarray*}
E_{k}^{'} & := & \{(\hat{\tau}_{k+1}-\tau_{k})\ge d_{\min}/2\}\;,\\
E_{k}^{''} & := & \{(\tau_{k}-\hat{\tau}_{k})\ge d_{\min}/2\}\;.
\end{eqnarray*}
In particular, one can demonstrate that:

\begin{equation}
P[A_{T,k}\cap D_{T}^{(m)}]\le P[A_{T,k}\cap E_{k}^{'}\cap D_{T}^{(m)}]+\sum_{j=k+1}^{K}P[E_{j}^{''}\cap E_{j}^{'}\cap D_{T}^{(m)}]\;.\label{eq:compliment_middle_split-1}
\end{equation}

\end{lemma*}

\begin{proof}
The strategy is to expand the probability in terms of exhaustive events
(relating to the estimated changepoint positions), under a symmetry
argument, we assume $\hat{\tau}_{k}<\tau_{k}$. Noting $P[E_{k}^{'}\cup E_{k+1}^{''}]=1$,
then expanding the original event, we find

\begin{eqnarray*}
P[A_{T,k}\cap D_{T}^{(m)}] & \le & P[A_{T,k}\cap D_{T}^{(m)}\cap E_{k}^{'}]+P[A_{T,k}\cap D_{T}^{(m)}\cap E_{k+1}^{''}]\\
 & \le & P[A_{T,k}\cap D_{T}^{(m)}\cap E_{k}^{'}]+P[D_{T}^{(m)}\cap E_{k+1}^{''}]\;.
\end{eqnarray*}
Now consider the event $D_{T}^{(m)}\cap E_{k+1}^{''}$ corresponding
to the second term. One can then expand the probability of this intersection
over the events $E_{k+1}^{'}$ and $E_{k+2}^{''}$ relating to the
next changepoint, i.e
\[
P[D_{T}^{(m)}\cap E_{k}^{''}]\le P[D_{T}^{(m)}\cap E_{k+1}^{''}\cap E_{k+1}^{'}]+P[D_{T}^{(m)}\cap E_{k+1}^{''}\cap E_{k+2}^{''}]\;.
\]
Again, $P[D_{T}^{(m)}\cap E_{k}^{''}\cap E_{k+2}^{''}]$ may be upper
bounded by $P[D_{T}^{(m)}\cap E_{k+2}^{''}]$ such that $P[D_{T}^{(m)}\cap E_{k}^{''}\cap E_{k+2}^{''}]\le P[D_{T}^{(m)}\cap E_{k+2}^{''}]$.
Cascading this over all changepoints $j=k+1,\ldots,K$ we have
\[
P[D_{T}^{(m)}\cap E_{k}^{''}]\le\sum_{j=k+1}^{K}P[D_{T}^{(m)}\cap E_{j}^{''}\cap E_{j+1}^{'}]\;.
\]

\end{proof}

\subsection{Proof of Lemma \ref{lemma:union_bound_left}}

\label{subsec:proof_lemma_union_bound_left}

\begin{lemma*}
The probability of $D_{T}^{(l)}$ is bounded by 
\[
P[D_{T}^{(l)}]\le2^{K}\sum_{k=1}^{K-1}\sum_{l\ge k}^{K-1}P[E_{l}^{''}\cap E_{l}^{'}]+2^{K}P[E_{K}^{'}]\;.
\]

\end{lemma*}

\begin{proof}
We present here an expanded version of the proof given in \citet{Harchaoui2010}. Recall the definitions of the different events:
\[
E_{k}^{'}:=\{(\hat{\tau}_{k+1}-\tau_{k})\ge\frac{d_{\min}}{2}\}\quad\mathrm{and\quad}E_{k}^{''}:=\{(\tau_{k}-\hat{\tau}_{k})\ge\frac{d_{\min}}{2}\}\;.
\]
For each new changepoint in the model, there is an extra option for
this (latest changepoint) to trigger the event 
\begin{equation}
\left\{ \exists k\in[K],\;\hat{\tau}_{k}\le\tau_{k-1}\right\} \;.\label{eq:left_compliment}
\end{equation}
In particular, the total number of combinations (of changepoints)
which could trigger this event doubles on the addition of an extra
changepoint. Lemma \ref{lemma:union_bound_left} considers the probability
of each of the changepoints being estimated to the left of $\tau_{k-1}$.
To start, we remark that the probability of $D_{T}^{(l)}$ is bounded
by 
\begin{equation}
P[D_{T}^{(l)}]\le\sum_{k=1}^{K}2^{k-1}P[\max\{l\in[K]\:|\:\hat{\tau}_{l}\le\tau_{l-1}\}=k]\;.\label{eq:union_bound_left}
\end{equation}
The term $P[\max\{l\in[K]\:|\:\hat{\tau}_{l}\le\tau_{l-1}\}=k]$ describes
the probability that the last changepoint (such that $\hat{\tau}_{l}$
is to the left, i.e. before $\tau_{l-1}$) is $k$. On increasing
$k$ by one (for $k\ge2$), the number of combinations of left/right
estimates for previous changepoints doubles. For example, consider
the case for $k=3$ such that the event $S_{3}:=\{\hat{\tau}_{3}\le\tau_{2}\}$
is triggered, see Fig. \ref{fig:compliment_left_combination}. The
possible results for previous changepoints are then $S_{2}:=\{\hat{\tau}_{2}\le\tau_{1}\}$,
it's compliment $S_{2}^{c}$, and the event $S_{1}:=\{\hat{\tau}_{1}\le1\}$
or $S_{1}^{c}$. In total, there are $2^{2}$ ways that the event
$S_{3}$ can occur\footnote{Arguably, there are actually 3 combinations of changepoint event that
can cause $S_{3}$ as $\hat{\tau}_{1}>\hat{\tau}_{0}=1$ by definition.
However, this does not effect the upper bound.}. In general for the changepoint $k$ there are $2^{k-1}$ combinations
of events that allow $S_{k}$ to be triggered. However, since these
events are not mutually exclusive, this only provides an upper bound.

\begin{figure}[h]
\begin{centering}
\includegraphics[width=0.7\columnwidth]{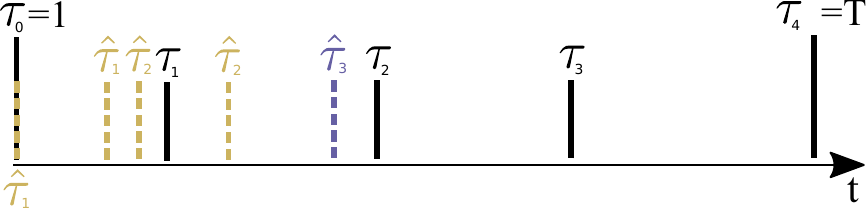}
\par\end{centering}
\caption{The gold changepoint estimates indicate examples of allowable positions
for the changepoints $l<k=3$ which satisfy $\{\hat{\tau}_{l}\le\tau_{l-1}\}$.
Note: for the case displayed $K=3$ and $k=3$ thus there are 4 combinations
of changepoints (in gold) that permit the purple event $\max\{l\in[4]\:|\:\hat{\tau}_{l}\le\tau_{l-1}\}=3$.
\label{fig:compliment_left_combination}}
\end{figure}
\citet{Harchaoui2010} and \citet{Kolar2012} remark that an event
where the $k$th changepoint is the largest to satisfy $\{\hat{\tau}_{l}\le\tau_{l-1}\}$,
is a subset of events relating to later changepoints $l\ge k$. Correspondingly,
we have
\begin{equation}
\{\max\{l\in[K]\:|\:\hat{\tau}_{l}\le\tau_{l-1}\}=k\}\subseteq\cup_{l=k}^{K}\{\tau_{l}-\hat{\tau}_{l}\ge d_{\min}/2\}\cap\{\hat{\tau}_{l+1}-\tau_{l}\ge d_{\min}/2\}\;.\label{eq:set_expansion_left_compliment}
\end{equation}

The union bound applied to (\ref{eq:set_expansion_left_compliment})
provides us with the bound: 
\[
P[\max\{l\in[K]\:|\:\hat{\tau}_{l}\le\tau_{l-1}\}=k]\le\sum_{l\ge k}P[\{\tau_{l}-\hat{\tau}_{l}\ge\frac{d_{\min}}{2}\}\cap\{\hat{\tau}_{l+1}-\tau_{l}\ge\frac{d_{\min}}{2}\}]\;,
\]
and thus
\[
P[D_{T}^{(l)}]\le\sum_{k=1}^{K}2^{k-1}\sum_{l\ge k}^{K}P[\{\tau_{l}-\hat{\tau}_{l}\ge\frac{d_{\min}}{2}\}\cap\{\hat{\tau}_{l+1}-\tau_{l}\ge\frac{d_{\min}}{2}\}]\;.
\]
Since we want an upper bound, the largest factor ($2^{K}$) can be
taken out the summation. The term $k=K$ contains the event $\{\hat{\tau}_{K+1}-\tau_{K}\ge d_{\min}/2\}$,
this occurs with probability one as the last changepoint $\hat{\tau}_{K+1}=T+1$.
We can thus truncate the final term and obtain the bound:
\begin{eqnarray*}
P[D_{T}^{(l)}] & \le & 2^{K}\sum_{k=1}^{K-1}\sum_{l\ge k}^{K-1}P[\{\tau_{l}-\hat{\tau}_{l}\ge\frac{d_{\min}}{2}\}\cap\{\hat{\tau}_{l+1}-\tau_{l}\ge\frac{d_{\min}}{2}\}]\\
 &  & +2^{K}P[\{\tau_{K}-\hat{\tau}_{K}\ge\frac{d_{\min}}{2}\}]\;.
\end{eqnarray*}
The above can be written in a shortened form by relating it to the
events $E_{k}^{'},E_{k}^{''}$ defined in (\ref{eq:Events_compliment_bound_k+1}),
such that
\[
P[D_{T}^{(l)}]\le2^{K}\sum_{k=1}^{K-1}\sum_{l\ge k}^{K-1}P[E_{l}^{''}\cap E_{l}^{'}]+2^{K}\mathbb{P}[E_{K}^{''}]\;.
\]

\end{proof}


\section{Proof of Model-Selection Consistency}

\subsection{Proof Overview}

Let us define a set of pairs $\mathcal{M}_{k}=\{(i,j)\:|\:\Theta_{0;ij}^{(k)}\ne0\}$
to be the support set of the true model in block $k$ and let $\mathcal{M}_{k}^{\perp}=\{(i,j)\:|\:\Theta_{ij}^{(k)}=0\}$
be its compliment. 

\begin{theorem*}{Bounds on the estimation error}

Consider the GFGL estimator with Assumptions (\ref{ass:bounded_cp_error},\ref{ass:incoherence}). Assume $\lambda_1=16\alpha^{-1}\delta$, $\lambda_2=\rho\lambda_1$ for some finite $\rho>0$ and $\delta>2^4\sqrt{2}c_{\sigma_\infty}\sqrt{\log(4p^{\beta}) /\gamma_{\min} T}$ where $\beta>2$. Let $d$ be the maximum degree of each node in the true model, and define  
\[
v_{\mathcal{C}}=6\{1+16\alpha^{-1}(1+2\rho)\}d\max\{ K_{\Sigma_{0}}K_{\Gamma_{0}},K_{\Gamma_{0}}^2 K_{\Sigma_{0}}^3 \}\;.
\]
Then for $T\ge 2^9 \gamma_{\min}^{-1}\max\{1/8c_{\sigma_\infty},v_{\mathcal{C}}\}^2 c^2_{\sigma_\infty}\log(4p^\beta)$, we have
\[
\|\hat{\Theta}^{(k)}-\Theta^{(k_{\max})}_0\|_\infty\le 2K_{\Gamma_{0}}\{1+16\alpha^{-1}(1+2\rho \hat{n}_k^{-1}) \}\delta\;,
\]
in probability greater than $1-1/p^{\beta-2}$.

\end{theorem*}

\begin{proof}

A popular and fairly general approach to demonstrating such consistency
is known as the \emph{primal-dual witness} method \citep{Wainwright2009}.
Principally, this method works by deconstructing the KKT conditions
of an M-estimator into two blocks. Let us label these conditions $\mathrm{KKT}(\mathcal{M},\partial r(\Theta_{\mathcal{M}}))$
and $\mathrm{KKT}(\mathcal{M}^{\perp},\partial r(\Theta_{\mathcal{M}^{\perp}}))$,
such that they respectively concern components of the the true model
$\Theta_{0;\mathcal{M}}$ and the compliment $\Theta_{0;\mathcal{M}^{\perp}}$.
The primal-dual witness approach consists of the following steps:
\begin{enumerate}
\item Solve a restricted problem; $\bar{\Theta}:=\arg\min_{\Theta}l(\Theta;X)+\lambda_{T}r(\Theta)$,
such that $\Theta_{\mathcal{M}^{\perp}}=0$. This constitutes a restricted
estimation problem, whereby the estimate is only supported on the
true model-subspace. It verifies that the $\mathrm{KKT}(\mathcal{M},\partial r(\Theta_{\mathcal{M}}))$
is satisfied under the block corresponding to the true support.
\item Select $\bar{R}$ as the sub-differential of the regulariser $r(\cdot)$
evaluated at $\bar{\Theta}$. Find the subgradient over the components
outside the model-subspace $\mathcal{M}^{\perp}$ via $\mathrm{KKT}(\bar{\Theta},\bar{R})$
\item Check that the sub-gradient in step (2) is sufficiently small to demonstrate
the solution is dual feasible.
\end{enumerate}
In what follows we wll dissect the GFGL estimator according to the
above steps, a similar approach in the stationary i.i.d setting is
discussed in \citet{Ravikumar2011}. In our case, we will require
some care to take account of the smoothing regulariser. 

Let $\hat{S}_{k;\hat{n}_k}:=\hat{n}_k^{-1}\sum_{t=1}^{\hat{n}_k} X^{(t)}(X^{(t)})^{\top}$
represent the empirical covariance matrix calculated by taking $\hat{n}_k$ samples from $X^{(t)}$ with covariance matrix $\Sigma_0^{(k)}$.
Since $\{\hat{\Theta}^{(k)}\}_{k=1}^{B}$ is an optimal solution
for GFGL, for each estimated block $k,l=1,\ldots,\hat{B}=K+1$ it needs
to satisfy
\begin{equation}
\sum_{l\ne k\in[\hat{B}]}\hat{n}_{lk}(W_{l;\hat{n}_{lk}})+\hat{n}_{kk}W_{k;\hat{n}_{kk}}-\hat{n}_{k}\hat{\Sigma}^{(k)}+\lambda_{1}\hat{n}_{k}\hat{R}_{1}^{(\hat{\tau}_{k-1})}+\lambda_{2}(\hat{R}_{2}^{(\hat{\tau}_{k-1})}-\hat{R}_{2}^{(\hat{\tau}_{k})})=0\;,\label{eq:block_optimality}
\end{equation}
where $\hat{n}_{lk}$ describes the proportion of overlap between
the $l$th true block and the $k$th estimated block, and $W_{k;n}:=\Sigma_{0}^{(k)}-\hat{S}_{k;n}$
represents the error accrued in the empirical covariance estimate.
The term $\sum_{l\ne k\in[\hat{B}]}\hat{n}_{lk}(W_{l;\hat{n}_{lk}})$
can be though of as providing a sampling bias due to estimation error in the changepoints, wheras the term $\hat{n}_{kk}W_{k;\hat{n}_{kk}}$
compares samples and the ground-truth of the same underlying covariance
matrix.

We will now proceed to construct an oracle estimator $\bar{\Theta}\in\tilde{\mathbb{R}}^{\hat{B}p\times\hat{B}p}$.
The oracle is constructed through solving the restricted problem
\begin{align*}
\bar{\Theta}:= \argmin_{\{U^{(k)}\succ 0|\:U_{\mathcal{M}^{\perp}}^{(k)}=0\}_{k=1}^{\hat{B}}} & \bigg[\sum_{k=1}^{\hat{B}}\big\{\sum_{l=1}^{\hat{B}}\hat{n}_{lk}\mathrm{tr}(\hat{S}^{(l)}U^{(k)})-\hat{n}_{k}\log\det(U^{(k)})\big\}\\
 & \quad\quad\quad+\lambda_{1}\sum_{k=1}^{\hat{B}}\hat{n}_{k}\|U^{(k)}\|_{1}+\lambda_{2}\sum_{k=2}^{\hat{B}}\|U^{(k)}-U^{(k-1)}\|_{F}\bigg]\;.
\end{align*}
The construction above does not utilise oracle knowledge to enforce
changepoint positions, only the sparsity structure of the block-wise
precision matrices. Again, for each estimate block, we obtain a set
of optimality conditions like (\ref{eq:block_optimality}). Let us
denote the sub-gradient of the restricted problem evaluated at the
oracle solution as $\bar{R}_{1}^{(k)}\equiv\bar{R}_{1}^{(\hat{\tau}_{k-1})}$
for the $\ell_{1}$ penalty, and $\bar{R}_{2}^{(\hat{\tau}_{k-1})},\bar{R}_{2}^{(\hat{\tau}_{k})}$
for the smoothing components. By definition the matrices $\bar{R}_{2}^{(\hat{\tau}_{k-1})},\bar{R}_{2}^{(\hat{\tau}_{k})}$
are members of the sub-differential and hence dual feasible. To show
that $\bar{\Theta}$ is also a minimiser of the unrestricted GFGL
problem (\ref{eq:GFGL_cost_function}), we will show that $\|\bar{R}_{1;\mathcal{M}^{\perp}}^{(k)}\|_{\infty}\le1$
and is hence dual-feasible.

\citet{Ravikumar2011} Lemma 4 demonstrates that for the standard
graphical lasso problem strict dual-feasibiility can be obtained
by bounding the max of the sampling and estimation error. The estimation
error (on the precision matrices) is tracked through the difference
(remainder) between the gradient of the log-det loss function and
its first-order Taylor expansion. In our case we will track the precision
matrices at each block $k$ via the \emph{remainder function} defined
as 
\[
\mathcal{E}(\Delta)=\bar{\Theta}^{-1}-\Theta_{0}^{-1}+\Theta_{0}^{-1}\Delta\Theta_{0}^{-1}\;,
\]
where $\Delta=\bar{\Theta}-\Theta_{0}\in\mathbb{R}^{p\times p}$.

\begin{lemma}
\label{lemma:dual-feasibility}
The out-of-subspace parameters are dual feasible such that $\|\bar{R}_{1;\mathcal{M}^{\perp}}^{(k)}\|_{\infty}<1$
if 
\begin{equation}
\label{eq:dual-feasibility-conditions}
\max\left\{ \|\mathrm{ave}(W^{(k)})\|_{\infty},\|\mathcal{E}(\Delta)\|_{\infty},\frac{\lambda_{2}}{\hat{n}_{k}}\|\bar{R}_{2}^{(\hat{\tau}_{k-1})}\|_{\infty},\frac{\lambda_{2}}{\hat{n}_{k}}\|\bar{R}_{2}^{(\hat{\tau}_{k})}\|_{\infty}\right\} \le\alpha\lambda_{1}/16\;,
\end{equation}
where $\mathrm{ave}(W^{(k)}):=\hat{n}_{k}^{-1}(\sum_{l\ne k}^{\hat{B}}\hat{n}_{lk}W_{l;\hat{n}_{lk}}+\hat{n}_{kk}W_{k;\hat{n}_{kk}})$.
\end{lemma}

We note at this point, that the condition (\ref{eq:dual-feasibility-conditions}) in the setting where $T\rightarrow\infty$ converges
to that of the standard graphical lasso \citep{Ravikumar2011}. Specifically, if changepoint error is bounded according to the event $E_\tau:=\{\max_{k}|\hat{\tau}_{k}-\tau_{k}|\le T\delta_{T}\}$, the mis-specification error averaged across the block converges
to the exact case $\mathrm{ave}(W^{(k)})\rightarrow W_{\mathrm{exact}}^{(k)}$.
%
%
To make this argument more specific, we construct a loose bound on the sampling error accumulated over an estimated block.

\begin{lemma}
\label{lemma:bound_on_sampling_average}
The sampling error over a block is almost surely bounded according
to
\[
\sum_{l\in\hat{\mathcal{B}}^{(k)}}\hat{n}_{lk}\|W_{l;\hat{n}_{lk}}\|_{\infty}\le\max\{\hat{n}_{k},d_{\min}\}\|W_{l_\infty;d_{\min}/2}\|_{\infty}\;,
\]
and thus the average sampling error is bounded according to
\[
\|\mathrm{ave}(W^{(k)})\|_{\infty}\le\max\{1,d_{\min}/\hat{n}_{k}\}\|W_{\infty;d_{\min}/2}\|_{\infty}\;.
\]
\end{lemma}

If changepoint estimation is consistent (according to Assumption \ref{ass:bounded_cp_error})
then $P(E_\tau)=f_{\tau}(T,p)$, and thus have
$d_{\min}/\hat{n}_{k}<d_{\min}/(d_{\min}-\delta_{T}T)\rightarrow1$,
as $T\rightarrow\infty$. As a result, we bound $\|\mathrm{ave}(W^{(k)})\|_{\infty}\le\|W_{\infty;d_{\min}/2}\|_{\infty}$
from above and then analyse the conditions (\ref{eq:dual-feasibility-conditions}) on condition of the intersection $E_\tau \cap \mathcal{C}$ where
\[
\mathcal{C}:=\{ \|\hat{W}_{k_{\infty};d_{\min}/2}\|_\infty \le \delta \}\;.
\]
Through choice of regulariser $\lambda=16\alpha^{-1}\delta$, the condition $\|\mathrm{ave}(W^)\|_\infty\le \alpha\lambda_1/16$ is automatically satisfied. We now turn our attention to the size of the remainder $\|\mathcal{E}(\Delta)\|_\infty$. In the first step, we directly invoke a result from \cite{Ravikumar2011}:

\begin{lemma}
\label{lemma:bound_on_remainder}
If the bound $\|\Delta\|_{\infty}\le(3K_{\Sigma_{0}}d)^{-1}$ holds
and $d$ is the maxmimum node degree, then
\[
\|\mathcal{E}(\Delta)\|_{\infty}\le\frac{3}{2}d\|\Delta\|_{\infty}^{2}K_{\Sigma_{0}}^{3}\;.
\]
\end{lemma}
While we can use the same relation as Ravikumar to map $\|\Delta\|_{\infty}$
to $\|\mathcal{E}(\Delta)\|_{\infty}$ we need to modify our argument
for the actual control on $\|\Delta\|_{\infty}$. 

\begin{lemma}
\label{lemma:bound_on_infty_error}
The elementwise $\ell_{\infty}$ norm of the error is bounded such
that $\|\bar{\Delta}\|_{\infty}=\|\bar{\Theta}-\Theta_{0}\|_{\infty}\le r$
if 
\begin{equation}
\label{eq:r_definition}
r:=2K_{\Gamma_{0}}\{\|\mathrm{ave}(W_{k})\|_{\infty}+\lambda_{1}+\lambda_{2}\hat{n}_{k}^{-1}(\|\bar{R}_{2}^{(\hat{\tau}_{k-1})}\|_{\infty}+\|\bar{R}_{2}^{(\hat{\tau}_{k})}\|_{\infty})\}\;,
\end{equation}
and $
r \le\min \{(3K_{\Sigma_{0}}d)^{-1},(3K_{\Sigma_{0}}^{3}K_{\Gamma_{0}}d)^{-1} \}$.

\end{lemma}

We now propogate the results of Lemma \ref{lemma:bound_on_infty_error} through Lemma \ref{lemma:bound_on_remainder}, while conditioning on event $\{\mathcal{C}\cap E_\tau\}$.
First, let us note the contribution of the fused sub-gradient is bounded $\lambda_2\hat{n}_k^{-1}(\|\bar{R}_{2}^{(\hat{\tau}_{k-1})}\|_{\infty}+\|\bar{R}_{2}^{(\hat{\tau}_{k})}\|_{\infty})\le 2\lambda_2\hat{n}^{-1}_k$. Let us further assume that $\lambda_2=\lambda_1\rho$ for $\rho>0$, we now upper bound (\ref{eq:r_definition}) with the stated form of $\lambda_1$  such that
\[
r\le r_{\mathcal{C}}:=2K_{\Gamma_{0}}\{\delta+\lambda_{1}(1+2\rho\hat{n}_{k}^{-1})\}=2K_{\Gamma_{0}}\{1+16\alpha^{-1}(1+2\rho \hat{n}_k^{-1}) \}\delta \;.
\]
The condition in Lemma \ref{lemma:bound_on_infty_error} is now met, if $\delta\in(0,1/\max\{1/8c_{\sigma_\infty},v_{\mathcal{C}}\})$ where 
\[
v_{\mathcal{C}}=6\{1+16\alpha^{-1}(1+2\rho\hat{n}_{k}^{-1})\}d\max\{ K_{\Sigma_{0}}K_{\Gamma_{0}},K_{\Gamma_{0}}^2 K_{\Sigma_{0}}^3 \}\;.
\]
%
Using Lemma \ref{proof:empirical_covariance_frobenius_bound} we have the probabilistic bound on $\mathcal{C}^c$ given as
$
P(\|W_{k_{\infty};d_{\min}/2}\|_\infty>\delta)\le 1/p^{(\beta-2)}
$,
where we need $\delta>2^4\sqrt{2}c_{\sigma_\infty}\sqrt{\log (4p^{\beta})/\gamma_{\min}T}$. Remember $c_{\sigma_\infty}=\max_k\{c_{\sigma_k}\}$, i.e. we assume the slowest concentration possible over all blocks $k=1,\ldots,B$. This results in a lower bound for the  sample size of
\begin{equation}
T\ge 2^9 \gamma_{\min}^{-1}\max\{1/8c_{\sigma_\infty},v_{\mathcal{C}}\}^2 c^2_{\sigma_\infty}\log(4p^\beta)\quad;\quad\beta>2\;.
\end{equation}
The restricted solution is hence dual-feasible under the conditions of the proof such that $\bar{\Theta}=\hat{\Theta}$. The true edges are included in the estimate with probability at least $1-1/p^{\beta-2}$. The final step of the proof is to demonstrate that all possible solutions to GFGL maintain this relation. In the case of GFGL, the following lemma states that the objective function is strictly convex on the positive-definite cone. Hence, if we find a minima it is the global minima, and the dual-feasibility condition ensures that the suggested bounds are achieved.
\begin{lemma}
\label{lemma:strict_convexity}
For matrices $\Theta_{T}\in\mathcal{S}_{++}^{T}:=\{\{A^{(t)}\}_{t=1}^{T}\:|\:A^{(t)}\succ0\:,\:A^{(t)}=A^{(t)\top}\}$
the GFGL cost function is strictly convex.
\end{lemma}
\end{proof}

\begin{theorem*}{Model-selection consistency}

In addition to the assumptions in Theorem \ref{thm:bound_on_estimation error}. Let $\theta_{\min}^{(k)} := \min_{ij}|\Theta^{(k)}_{0;ij}|$ for all $(i,j)\in\mathcal M_k$ and for each $k=1,\ldots,B$. Let
\[
v_{\theta} = 2K_{\Sigma_0}\{1+16\alpha^{-1}(1+2\rho\hat{n}_k^{-1})\}\theta_{\min}^{-1}\;.
\]
If $T\ge 2^9\gamma_{\min}^{-1}\max\{(8c_{\sigma_\infty})^{-1},v_\mathcal{C},v_{\theta} \}^2c_{\sigma_\infty}^2\log(4p^\beta)$ then GFGL attains sign-consistency with probability tending to one:
\[
P\{E_{\mathcal{M}}(\hat{\Theta}^{(k)};\Theta_{0}^{(k_{\max})})\}\ge 1-1/p^{\beta-2}\rightarrow 1\quad;\quad\mathrm{as}\;(p,T)\rightarrow\infty\;.
\]
\end{theorem*}
\begin{proof}
This result is a simple corollary of Theorem \ref{thm:bound_on_estimation error} with the additional condition on a known and appropriately sized $\theta_{\min}$. The proof follows from Lemma \ref{lemma:bound_on_infty_error} where we require $r=r_{\theta} \ge \theta_{\min}/2$. The element-wise error incurred by the estimate $|\hat{\Theta}_{ij}^{(k)}-\Theta^{(k)}_0|$ is not enough to change the sign of the estimate and the result follows. An exact analogy is given (in the static case) by Lemma 7 \citet{Ravikumar2011}.
\end{proof}

\subsection{Dual-feasibility with sampling and mis-specification error (Proof of Lemma \ref{lemma:dual-feasibility})}

\begin{lemma*}{Dual Feasibility}

The out-of-subspace parameters are dual feasible such that $\|\bar{R}_{1;\mathcal{M}^{\perp}}^{(k)}\|_{\infty}<1$
if 
\[
\max\left\{ \|\mathrm{ave}(W)\|_{\infty},\|\mathcal{E}(\Delta)\|_{\infty},\frac{\lambda_{2}}{\hat{n}_{k}}\|\bar{R}_{2}^{(\hat{\tau}_{k-1})}\|_{\infty},\frac{\lambda_{2}}{\hat{n}_{k}}\|\bar{R}_{2}^{(\hat{\tau}_{k})}\|_{\infty}\right\} \le\alpha\lambda_{1}/16\;.
\]

\end{lemma*}

\begin{proof}

We can write the block-wise optimality conditions (\ref{eq:block_optimality})
for the restricted estimator as
\begin{align*}
 & (\Theta_{0}^{(k)})^{-1}\Delta^{(k)}(\Theta_{0}^{(k)})^{-1}-\mathcal{E}(\Delta^{(k)})+\frac{1}{\hat{n}_{k}}\left(\sum_{l\ne k}^{\hat{B}}\hat{n}_{lk}W_{l;\hat{n}_{lk}}+\hat{n}_{kk}W_{k;\hat{n}_{kk}}\right)\\
 & \quad+\lambda_{1}\bar{R}_{1}^{(k)}+\frac{\lambda_{2}}{\hat{n}_{k}}(\bar{R}_{2}^{(\hat{\tau}_{k-1})}-\bar{R}_{2}^{(\hat{\tau}_{k})})=0\;.
\end{align*}

As pointed out in \citet{Ravikumar2011}, this equation may be written
as an ordinary linear equation by vectorising the matrices, for instance
$\mathrm{vec}\{(\Theta_{0}^{(k)})^{-1}\Delta^{(k)}(\Theta_{0}^{(k)})^{-1}\}=\{(\Theta_{0}^{(k)})^{-1}\otimes(\Theta_{0}^{(k)})^{-1}\}\mathrm{vec}(\Delta^{(k)})\equiv\Gamma_{0}\mathrm{vec}(\Delta)$.
Utilising the fact $\Delta_{\mathcal{M}^{\perp}}=0$ we can split
the optimality conditions into two blocks of linear equations
\begin{align}
\Gamma_{0;\mathcal{M}\mathcal{M}}^{(k)}\mathrm{vec}(\Delta_{\mathcal{M}}^{(k)})+\mathrm{vec}(G_{\hat{n}_{k}}^{(k)}(X;\lambda_{1},\lambda_{2})_{\mathcal{M}}) & =0\;\label{eq:optimaility_in_model}\\
\Gamma_{0;\mathcal{M}^{\perp}\mathcal{M}}^{(k)}\mathrm{vec}(\Delta_{\mathcal{M}}^{(k)})+\mathrm{vec}(G_{\hat{n}_{k}}^{(k)}(X;\lambda_{1},\lambda_{2})_{\mathcal{M}^{\perp}}) & =0\;,\label{eq:optimality_out_of_model}
\end{align}
where
\[
G(X;\lambda_{1},\lambda_{2}):=\mathrm{ave}(W^{(k)})-\mathcal{E}(\Delta^{(k)})+\lambda_{1}\bar{R}_{1}^{(k)}+\hat{n}_{k}^{-1}\lambda_{2}(\bar{R}_{2}^{(\hat{\tau}_{k-1})}-\bar{R}_{2}^{(\hat{\tau}_{k})})\;,
\]
and the average empirical covariance error over the block is given
by

\[
\mathrm{ave}(W^{(k)}):=\hat{n}_{k}^{-1}(\sum_{l\ne k}^{\hat{B}}\hat{n}_{lk}W_{l;\hat{n}_{lk}}+\hat{n}_{kk}W_{k;\hat{n}_{kk}})\;.
\]
Solving (\ref{eq:optimaility_in_model}) for $\mathrm{vec}(\Delta_{\mathcal{M}}^{(k)})$
we find
\[
\mathrm{vec}(\Delta_{\mathcal{M}}^{(k)})=-(\Gamma_{0;\mathcal{M}\mathcal{M}}^{(k)})^{-1}\mathrm{vec}\{G(X;\lambda_{1},\lambda_{2})_{\mathcal{M}}\}\;.
\]
Substituting this into (\ref{eq:optimality_out_of_model}) and re-arranging
for $\bar{R}_{1;\mathcal{M}^{\perp}}$ gives

\[
\mathrm{vec}(G_{\hat{n}_{k}}^{(k)}(X;\lambda_{1},\lambda_{2})_{\mathcal{M}^{\perp}})=\Gamma_{0;\mathcal{M}^{\perp}\mathcal{M}}^{(k)}(\Gamma_{0;\mathcal{M}\mathcal{M}}^{(k)})^{-1}\mathrm{vec}\{G(X;\lambda_{1},\lambda_{2})_{\mathcal{M}}\}\;,
\]
and thus

\begin{align*}
\bar{R}_{1;\mathcal{M}^{\perp}}^{(k)} & =\frac{1}{\lambda_{1}}H\mathrm{vec}\{\mathrm{ave}(W_{\mathcal{M}})-\mathcal{E}_{\mathcal{M}}(\Delta)\}+\frac{\lambda_{2}}{\hat{n}_{k}\lambda_{1}}H\mathrm{vec}\{(\bar{R}_{2}^{(\hat{\tau}_{k-1})}-\bar{R}_{2}^{(\hat{\tau}_{k})})_{\mathcal{M}}\}+H\mathrm{vec}(\bar{R}_{1;\mathcal{M}}^{(k)})\\
 & \quad-\mathrm{\frac{1}{\lambda_{1}}vec}\{\mathrm{ave}(W_{\mathcal{M}^{\perp}})-\mathcal{E}_{\mathcal{M}^{\perp}}(\Delta^{(k)})\}-\frac{\lambda_{2}}{\hat{n}_{k}\lambda_{1}}\mathrm{vec}\{(\bar{R}_{2}^{(\hat{\tau}_{k-1})}-\bar{R}_{2}^{(\hat{\tau}_{k})})_{\mathcal{M}^{\perp}}\}\;,
\end{align*}
where $H^{(k)}=\Gamma_{0;\mathcal{M}^{\perp}\mathcal{M}}^{(k)}(\Gamma_{0;\mathcal{M}\mathcal{M}}^{(k)})^{-1}$.
Taking the $\ell_{\infty}$ norm of both sides gives
\begin{align*}
\|\bar{R}_{1;\mathcal{M}^{\perp}}^{(k)}\|_{\infty} & \le\frac{1}{\lambda_{1}}\vertiii{H^{(k)}}_{\infty}(\|\mathrm{ave}(W_{\mathcal{M}})\|_{\infty}+\|\mathcal{E}_{\mathcal{M}}(\Delta)\|_{\infty})+\|H^{(k)}\mathrm{vec}(\bar{R}_{1;\mathcal{M}}^{(k)})\|_{\infty}\\
 & \quad+\frac{1}{\lambda_{1}}(\|\mathrm{ave}(W_{\mathcal{M}^{\perp}})\|_{\infty}+\|\mathcal{E}_{\mathcal{M}^{\perp}}(\Delta)\|_{\infty})\\
 & \quad+\frac{\lambda_{2}}{\hat{n}_{k}\lambda_{1}}\left\{ \vertiii{H^{(k)}}_{\infty}(\|\bar{R}_{2;\mathcal{M}}^{(\hat{\tau}_{k-1})}\|_{\infty}+\|\bar{R}_{2;\mathcal{M}}^{(\hat{\tau}_{k})}\|_{\infty})+\|\bar{R}_{2;\mathcal{M}^{\perp}}^{(\hat{\tau}_{k-1})}\|_{\infty}+\|\bar{R}_{2;\mathcal{M}^{\perp}}^{(\hat{\tau}_{k})}\|_{\infty}\right\} \;.
\end{align*}

\begin{lemma}

The error in the model-space dominates that outside such that
\begin{align}
\|\mathrm{ave}(W_{\mathcal{M}^{\perp}})\|_{\infty} & \le\|\mathrm{ave}(W_{\mathcal{M}})\|_{\infty}\;,\\
\|\mathcal{E}_{\mathcal{M}^{\perp}}(\Delta)\|_{\infty} & \le\|\mathcal{E}_{\mathcal{M}}(\Delta)\|_{\infty}\;.
\end{align}
 Furthermore, the maximum size of the sub-gradient in the model subspace
is bounded $\|\bar{R}_{1;\mathcal{M}}^{(k)}\|_{\infty}\le1\;.$

\end{lemma}

Via the results above, we obtain $\|H^{(k)}\mathrm{vec}(\bar{R}_{1;\mathcal{M}}^{(k)})\|_{\infty}\le1-\alpha$
and

\begin{align*}
\|\bar{R}_{1;\mathcal{M}^{\perp}}^{(k)}\|_{\infty} & \le\lambda_{1}^{-1}(2-\alpha)\{\|\mathrm{ave}(W_{\mathcal{M}})\|_{\infty}+\|\mathcal{E}_{\mathcal{M}}(\Delta)\|_{\infty}\\
 & \quad+\lambda_{2}\hat{n}_{k}^{-1}(\|\bar{R}_{2;\mathcal{M}}^{(\hat{\tau}_{k-1})}\|_{\infty}+\|\bar{R}_{2;\mathcal{M}}^{(\hat{\tau}_{k})}\|_{\infty})\}+\|H^{(k)}\mathrm{vec}(\bar{R}_{1;\mathcal{M}}^{(k)})\|_{\infty}\;.
\end{align*}
The condition stated in the lemma now ensures $\|\bar{R}_{1;\mathcal{M}^{\perp}}^{(k)}\|_{\infty}<1$.

\end{proof}

\subsection{Control on the remainder term $\|\mathcal{E}(\Delta)\|_{\infty}$ (Proof of Lemmas \ref{lemma:bound_on_remainder},\ref{lemma:bound_on_infty_error})}

\begin{lemma*}{Lemma 5 \cite{Ravikumar2011}}

If the bound $\|\Delta\|_{\infty}\le(3K_{\Sigma_{0}}d)^{-1}$ holds
and $d$ is the maxmimum node degree, then
\[
\|\mathcal{E}(\Delta)\|_{\infty}\le\frac{3}{2}d\|\Delta\|_{\infty}^{2}K_{\Sigma_{0}}^{3}\;.
\]

\begin{proof}

The reader is directed to \citet{Ravikumar2011} for full details.
The proof relies on representing the remainder term of the log-det
function as
\begin{equation}
\mathcal{E}(\Delta)=\Theta_{0}^{-1}\Delta\Theta_{0}^{-1}\Delta J\Theta_{0}^{-1}\;,\label{eq:remainder_J}
\end{equation}
for matrix 
\[
J:=\sum_{m=0}^{\infty}(-1)^{m}(\Theta_{0}^{-1}\Delta)^{m}\;.
\]
Given the stated control on $\|\Delta\|_{\infty}$ the norm of this
matrix can be bounded such that $\vertiii{J^{\top}}_{\infty}\le3/2$,
the result follows by working through (\ref{eq:remainder_J}) with
a maximum degree size $d$.

\end{proof}

\end{lemma*}

\begin{lemma*}{Control of Estimation error}

The elementwise $\ell_{\infty}$ norm of the error is bounded such
that $\|\bar{\Delta}\|_{\infty}=\|\bar{\Theta}-\Theta_{0}\|_{\infty}\le r$
if 
\[
r:=2K_{\Gamma_{0}}\{\|\mathrm{ave}(W_{k})\|_{\infty}+\lambda_{1}+\lambda_{2}\hat{n}_{k}^{-1}(\|\bar{R}_{2}^{(\hat{\tau}_{k-1})}\|_{\infty}+\|\bar{R}_{2}^{(\hat{\tau}_{k})}\|_{\infty})\}\;,
\]
and 
\[
r \le\min\left\{\frac{1}{3K_{\Sigma_{0}}d},\frac{1}{3K_{\Sigma_{0}}^{3}K_{\Gamma_{0}}d}\right\}\;.
\]

\end{lemma*}

\begin{proof}

Note that $\bar{\Theta}_{\mathcal{M}^{\perp}}=\Theta_{0;\mathcal{M}^{\perp}}=0$
and thus $\|\Delta\|_{\infty}=\|\Delta_{\mathcal{M}}\|_{\infty}$. We
follow \citet{Ravikumar2011} (Lemma 6) in the spirit of our proof.
The first step is to characterise the solution $\bar{\Theta}_{\mathcal{M}}$
in terms of its zero-gradient condition (of the restricted oracle
problem). Define a function to represent the block-wise optimality
conditions (akin to Eq. 75 \citet{Ravikumar2011})
\[
Q(\Theta_{\mathcal{M}}^{(k)})=-(\Theta_{\mathcal{M}}^{(k)})^{-1}+\left(\frac{1}{\hat{n}_{k}}\sum_{t=\hat{\tau}_{k-1}}^{\hat{\tau}_{k}-1}\hat{S}_{\mathcal{M}}^{(t)}\right)+\lambda_{1}\bar{R}_{1}^{(k)}+\frac{\lambda_{2}}{\hat{n}_{k}}(\bar{R}_{2}^{(\hat{\tau}_{k-1})}-\bar{R}_{2}^{(\hat{\tau}_{k})})=0\;.
\]
Now construct a map $F:\Delta_{\mathcal{M}}\mapsto F(\Delta_{\mathcal{M}})$
such that its fixed points are equivalent to the zeros of the gradient
expression in terms of $\Delta_{\mathcal{M}}$. To simplify the analysis,
let us work with the vectorised form and define the map 
\[
F(\mathrm{vec}(\Delta_{\mathcal{M}})):=-(\Gamma_{0;\mathcal{M}\mathcal{M}})^{-1}\mathrm{vec}\{Q(\Theta_{\mathcal{M}}^{(k)})\}+\mathrm{vec}(\Delta_{\mathcal{M}})\;,
\]
such that $F\{\mathrm{vec}(\Delta_{\mathcal{M}})\}=\mathrm{vec}(\Delta_{\mathcal{M}})$
iff $Q(\Theta_{0;\mathcal{M}}^{(k)}+\Delta_{\mathcal{M}})=Q(\Theta_{\mathcal{M}}^{(k)})=0$.

Now, to ensure all solutions that satisfy the zero gradient expression
may have their error bounded within the ball we demonstrate that $F$
maps a $\ell_{\infty}$ ball $\mathbb{B}(r):=\{\Theta_{\mathcal{M}}|\:\|\Theta_{\mathcal{M}}\|_{\infty}\le r\}$
onto itself. Expanding $F(\mathrm{vec}(\Delta_{\mathcal{M}}))$, we
find
\begin{align*}
F(\mathrm{vec}(\Delta_{\mathcal{M}})) & =-(\Gamma_{0;\mathcal{M}\mathcal{M}})^{-1}\mathrm{vec}\{Q(\Theta_{0;\mathcal{M}}^{(k)}+\Delta_{\mathcal{M}})\}+\mathrm{vec}(\Delta_{\mathcal{M}})\\
 & =(\Gamma_{0;\mathcal{M}\mathcal{M}})^{-1}\mathrm{vec}\big[\{(\Theta_{0}^{(k)}+\Delta)^{-1}-(\Theta_{0}^{(k)})^{-1}\}_{\mathcal{M}}-\mathrm{ave}(W_{k;\mathcal{M}})-\lambda_{1}\bar{R}_{1;\mathcal{M}}^{(k)}\\
 & \quad-\lambda_{2}\hat{n}_{k}^{-1}(\bar{R}_{2;\mathcal{M}}^{(\hat{\tau}_{k-1})}-\bar{R}_{2;\mathcal{M}}^{(\hat{\tau}_{k})})\big]+\mathrm{vec}(\Delta_{\mathcal{M}})\\
 & =T_{1}-T_{2}\;,
\end{align*}
where 
\begin{align*}
T_{1} & :=(\Gamma_{0;\mathcal{M}\mathcal{M}})^{-1}\mathrm{vec}\big[\{(\Theta_{0}^{(k)})^{-1}\Delta\}^{2}J(\Theta_{0}^{(k)})^{-1}\big]_{\mathcal{M}}\;,\\
T_{2} & :=(\Gamma_{0;\mathcal{M}\mathcal{M}})^{-1}\mathrm{vec}\big[\mathrm{ave}(W_{k;\mathcal{M}})+\lambda_{1}\bar{R}_{1;\mathcal{M}}^{(k)}+\lambda_{2}\hat{n}_{k}^{-1}(\bar{R}_{2;\mathcal{M}}^{(\hat{\tau}_{k-1})}-\bar{R}_{2;\mathcal{M}}^{(\hat{\tau}_{k})})\big]\;.
\end{align*}
The rest of the proof follows from \citet{Ravikumar2011}, where one
can show
\[
\|T_{1}\|_{\infty}\le\frac{3}{2}dK_{\Sigma_{0}}^{3}K_{\Gamma_{0}}\|\Delta\|_{\infty}^{2}\le\frac{3}{2}dK_{\Sigma_{0}}^{3}K_{\Gamma_{0}}r^{2}\;,
\]
under the assumptions of the lemma we obtain $\|T_{1}\|\le r/2$.
Combined with the stated form of $r$, we also find $\|T_{2}\|_{\infty}\le r/2$
and thus $\|F(\mathrm{vec}(\Delta_{\mathcal{M}}))\|_{\infty}\le r$.
Through the construction of $F$, we have $\|\Delta_{\mathcal{M}}\|_{\infty}\le r$
iff $Q(\Theta_{0;\mathcal{M}}^{(k)}+\Delta_{\mathcal{M}})=Q(\Theta_{\mathcal{M}}^{(k)})=0$
and since $Q(\bar{\Theta}_{\mathcal{M}}^{(k)})=0$ for any $\bar{\Theta}_{\mathcal{M}}^{(k)}$
we obtain $\|\bar{\Delta}_{\mathcal{M}}\|_{\infty}\le r$ where $\bar{\Delta}_{\mathcal{M}}:=\bar{\Theta}-\Theta_{0}$.
Finally, the existence of a solution $\bar{\Theta}_{\mathcal{M}}^{(k)}$
corresponding to $\mathrm{vec}(\bar{\Delta}_{\mathcal{M}})\in\mathbb{B}(r)$
is guranteed by Brouwer's fixed point theorem (cite).

\end{proof}

\subsection{Strict Convexity of GFGL (Proof of Lemma \ref{lemma:strict_convexity})}

\begin{lemma*}{The GFGL cost function is strictly convex}

For matrices $\Theta_{T}\in\mathcal{S}_{++}^{T}:=\{\{A^{(t)}\}_{t=1}^{T}\:|\:A^{(t)}\succ0\:,\:A^{(t)}=A^{(t)\top}\}$
the GFGL cost function is strictly convex.

\end{lemma*}

\begin{proof}

The neagtive log-det barrier $-\log\det(\Theta^{(t)})$ is strictly
convex on $\Theta^{(t)}\in\mathcal{S}_{++}^{1}$. While the frobenius
norm is stictly convex on a given matrix $\|A\|_{F}$ for $A\in\mathcal{S}_{++}^{1}$
it is not striclty convex when considering the mixed norm $\sum_{t=2}^{T}\|\Theta^{(t)}-\Theta^{(t-1)}\|_{F}$
for $\Theta^{(t)}\in\mathcal{S}_{++}^{(T)}$. However, due to Lagrangian
duality, we can re-write the GFGL problem as an explicitly constrained
problem
\begin{align*}
 & \min_{\Theta_{T}\in\mathcal{S}_{++}^{(T)}}\big\{\sum_{t=1}^{T}\big[\langle\Theta^{(t)},\hat{S}^{(t)}\rangle-\log\det(\Theta^{(t)})\big]\big\}\\
\mathrm{such\:that} & \quad\sum_{t=1}^{T}\|\Theta_{-ii}^{(t)}\|_{1}+\frac{\lambda_2}{\lambda_1}\sum_{t=2}^{T}\|\Theta^{(t)}-\Theta^{(t-1)}\|_{F}\le C(\lambda_{1})\;.
\end{align*}
We can alternatively write
\begin{align*}
 & \min_{\Theta_{T}\in\mathcal{S}_{++}^{(T)}}\big\{\sum_{t=1}^{T}\big[\langle\Theta^{(t)},\hat{S}^{(t)}\rangle-\log\det(\Theta^{(t)})\big]\big\}\\
\mathrm{such\:that} & \quad\sum_{t=1}^{T}\|\Theta_{-ii}^{(t)}\|_{1}\le C_{\mathrm{sparse}}(\lambda_{1},\lambda_2)\\
 & \quad\sum_{t=2}^{T}\|\Theta^{(t)}-\Theta^{(t-1)}\|_{F}\le C_{\mathrm{smooth}}(\lambda_{1},\lambda_2)\;.
\end{align*}

A similar argument to that used in \citet{Ravikumar2011} now holds.
Specifically, we note that even the rank one estimate $\hat{S}^{(t)}$
will have positive diagonal entries $\hat{S}_{ii}^{(t)}>0$ for all
$i=1,\ldots,p$. The off-diagonal entries in the precision matrix
are restricted through the $\ell_{1}$ term. Unlike in the standard
static case, the size of this norm is not related to just a single
precision matrix, rather it counts the size of the off-diagonals over
the whole set $\{\Theta^{(t)}\}_{t=1}^T$. Thus, to obtain strict convexity, one also needs to
include appropriate smoothing. To borrow the same argument as used
in \citet{Ravikumar2011}, we need to demonstrate that for any time-point
$t$ we can construct a problem of the form
\begin{align*}
 & \min_{\Theta^{(t)}\in\mathcal{S}_{++}^{(1)}}\big\{\langle\Theta^{(t)},\hat{S}^{(t)}\rangle-\log\det(\Theta^{(t)})\big\}\\
\mathrm{such\:that} & \quad\|\Theta_{-ii}^{(t)}\|_{1}\le C_{t}(\lambda_{1},\lambda_2)\;.
\end{align*}
The constraint due to smoothing allows exactly this, for instance,
one may obtain a bound
\[
\|\Theta_{-ii}^{(t)}\|_{1}\le C_{\mathrm{sparse}}(\lambda_{1},\lambda_2)-\sum_{s\ne t}\|\Theta_{-ii}^{(s)}\|_{1}\;.
\]
Writing $\Theta^{(s)}=\Theta^{(1)}+\sum_{q=2}^{s}(\Theta^{(q)}-\Theta^{(q-1)})$
for $s\ge2$ we obtain
\begin{align*}
\sum_{s\ne t}\|\Theta^{(s)}\|_{1} & \le\|\Theta^{(1)}\|_{1}+\sum_{s\ne t}\sum_{q=1}^{s}\|\Theta^{(q)}-\Theta^{(q-1)}\|_{1}\\
 & \le pC_{\mathrm{smooth}}(\lambda_{1},\lambda_2)+\|\Theta^{(1)}\|_{1}\;,
\end{align*}
where we note $\|\cdot\|_{1}\le p\|\cdot\|_{F}$. Converting to the
bound for the $\ell_{1}$ norm at time $t$ we find
\[
\|\Theta_{-ii}^{(t)}\|_{1}\le C_{\mathrm{sparse}}(\lambda_{1},\lambda_2)-pC_{\mathrm{smooth}}(\lambda_{1},\lambda_2)-\|\Theta^{(1)}\|_{1}\equiv C_{t}(\lambda_{1},\lambda_2)\;,
\]
and thus an effective bound on the time-specific $\ell_{1}$ norm
can be obtained.

\end{proof}

\subsection{Bounding the block-wise sampling error (Proof of Lemma \ref{lemma:bound_on_sampling_average})}

Let $\hat{\mathcal{B}}^{(k)}$ be a set containing all true block
indexes $l$ which overlap with the estimated block of interest $k$
and let $\hat{n}_{lk}=\min\{\min\{\hat{\tau}_{k},\tau_{l+1}\}-\max\{\hat{\tau}_{k-1},\tau_{l}\},0\}$
be the overlap of the $l$th true block and $k$th estimate block.
The average sampling error over block $k$ can then be written as

\[
\mathrm{ave}(W^{(k)}):=\frac{1}{\hat{n}_{k}}\sum_{l\in\mathcal{\hat{B}}^{(k)}}\hat{n}_{lk}(\hat{S}_{l;\hat{n}_{lk}}-\Sigma^{(l)})\;.
\]
Given a lower bound on the true size of blocks $d_{\min}=\min_{k}\{\tau_{k}-\tau_{k-1}\}$
we can bound the average sampling error (even in the case where $\delta_{T}T>d_{\min}$).

\begin{lemma*}
Let $\hat{n}_k:=\hat{\tau}_k-\hat{\tau}_{k-1}$. The sampling error over a block is almost surely bounded according
to

\[
\sum_{l\in\hat{\mathcal{B}}^{(k)}}\hat{n}_{lk}\|W_{l;\hat{n}_{lk}}\|_{\infty}\le\max\{\hat{n}_{k},d_{\min}\}\|W_{l_\infty;d_{\min}/2}\|_{\infty}\;,
\]
\end{lemma*}
\begin{proof}
We consider to bound cases where $\hat{\tau}_{k}-\hat{\tau}_{k-1}<d_{\min}$ differently according to
\[
\sum_{l\in\mathcal{B}^{(k)}}\hat{n}_{lk}\|W_{l}^{(\hat{n}_{lk})}\|_{\infty}\le\begin{cases}
(\hat{\tau}_{k}-\hat{\tau}_{k-1})\|W_{\infty}^{(d_{\min}/2)}\|_{\infty} & \mathrm{if}\;\hat{\tau}_{k}-\hat{\tau}_{k-1}\ge d_{\min}\;(A)\\
d_{\min}\|W_{\infty}^{(d_{\min}/2)}\|_{\infty} & \mathrm{if}\;\hat{\tau}_{k}-\hat{\tau}_{k-1}<d_{\min}\;(B)
\end{cases}
\]

\subsubsection*{Case A}
To show the first bound where $\hat{\tau}_{k}-\hat{\tau}_{k-1}\ge d_{\min}$, consider the simple case where we have two true blocks that overlap the estimated block $k$, namely $l$ and $l+1$. Note that $W_{l}^{(n_{1})}$ is independent of $W_{l+1}^{(n_{2})}$, thus we can write 
\[
P[n_{1}\|W_{l}^{(n_{1})}\|_{\infty}+n_{2}\|W_{l}^{(n_{2})}\|_{\infty}>\epsilon]=P[\|W_{l}^{(n_{1})}\|_{\infty}>\epsilon/2n_{1}]+P[\|W_{l+1}^{(n_{2})}\|_{\infty}>\epsilon/2n_{2}]\;.
\]
If we further assume that the structure in the blocks is the same, i.e. $\Sigma_{0}^{(l)}=\Sigma_{0}^{(l+1)}$ then we can use the bound of Ravikumar et al. (c.f. Lemma \ref{lemma:tail_bound_hd}) to state
\begin{equation}
\label{eq:n1_n2}
P[n_{1}\|W_{l}^{(n_{1})}\|_{\infty}+n_{2}\|W_{l}^{(n_{2})}\|_{\infty}>\epsilon]\propto p^{2}\big(\exp(-c_{l}^{-1}\epsilon^{2}n_{1}^{-1})+\exp(-c_{l}^{-1}\epsilon^{2}n_{2}^{-1})\big)\;.
\end{equation}
Letting $n_{2}=|\hat{\tau}_{k}-\hat{\tau}_{k-1}|-n_{1}$ and differentiating with respect to $n_{1}$ we find that the probability of exceeding $\epsilon$ is maximised at $n_{1}=n_{2}=|\hat{\tau}_{k}-\hat{\tau}_{k-1}|/2$. To arrive at the result (in the two block setting), we simply note that the block $\Sigma_{0}^{(l)}$ that we sample from should be taken to be the one with the largest $c_{\sigma}$.

\subsubsection*{Case B}
In this case $\hat{\tau}_{k}-\hat{\tau}_{k-1}<d_{\min}$, the interval over which the sampling error occurs is relatively small. We desire to find a multiplier in terms of $d_{\min}$ such that we can still maintain 
\begin{equation}
\label{eq:general_bound_small_interval}
\sum_{l\in B^{(k)}}\hat{n}_{lk}\|W_{l;\hat{n}_{lk}}\|_{\infty}\le n_{\mathrm{eff}}\|W_{l_\infty;d_{\min}/2}\|_{\infty}\;.
\end{equation}
In the extreme case where $\hat{n}_{lk}=1$ and $\hat{\tau}_{k}-\hat{\tau}_{k-1}=1$. From Ravikumar, we obtain
\[
P[\|W_{l;1}\|_{\infty}>\epsilon d_{\min}]\le p^{2}\exp(-c_{l}^{-1}\epsilon^{2}d_{\min}^{2})\;,
\]
and
\begin{equation}
\label{eq:limit_bound_small_interval}
P[\|W_{l;d_{\min}/2}\|_{\infty}>\epsilon]\le p^{2}\exp(-2c_{l}^{-1}\epsilon^{2}d_{\min}^{-1})\;,
\end{equation}
thus in high-probability $n_{\mathrm{eff}}=d_{\min}$ maintains the bound (\ref{eq:general_bound_small_interval}). From the bound on the first part (A) we noted that the worst case error was obtained at $n_{1}=n_{2}=\hat{\tau}_{k}-\hat{\tau}_{k-1}/2$. This result still holds in the case where $\hat{\tau}_{k}-\hat{\tau}_{k-1}<d_{\min}$. For all $2n_{1}=1,\ldots,d_{\min}$, we find
\[
P[\|W_{l;n_{1}}\|_{\infty}>\epsilon d_{\min}n_{1}^{-1}]\le p^{2}\exp(-c_{l}^{-1}\epsilon^{2}d_{\min}^{2}n_{1}^{-1})\;.
\]
The probability (\ref{eq:limit_bound_small_interval}) is obtained at the limit $n_{1}=d_{\min}/2$ and therefore the bound holds for all $n_{1},n_{2}<d_{\min}/2$.

\subsubsection*{Multiple Blocks}
For cases with more blocks the result still holds by noting that $d_{\max}>\min_{l\in K^{*}}\{n_{lk}\}$, thus the error from any blocks fully contained in the interval $\{\hat{\tau}_{k-1},\ldots,\hat{\tau}_{k}\}$ will be dominated by the blocks at either end. However, we do need to reassess the size of the boundary blocks $n_{1},n_{2}$. Rather than combining over an interval $|\hat{\tau}_{k}-\hat{\tau}_{k-1}|$ this interval is reduced according to 
\[
n_{1}=n_{2}=\frac{1}{2}\big(|\hat{\tau}_{k}-\hat{\tau}_{k-1}|-\sum_{l\in\tilde{\mathcal{B}}^{(k)}}|\tau_{l}-\tau_{l-1}|\big)\;,
\]
where $\tilde{\mathcal{B}}^{(k)}$ is the set $\mathcal{B}^{(k)}$ without the first and last elements (which represent lengths $n_{1}$ and $n_{2}$ respectively). From assuption we have a lower bound on the changepoint distance $d_{\min}$, we can thus upper bound the size of the boundary regions according to 
\[
n_{1}+n_{2}\le|\hat{\tau}_{k}-\hat{\tau}_{k-1}|-(|\mathcal{B}^{(k)}|-2)d_{\min}\;.
\]
Setting $n_{1}=n_{2}=\{|\hat{\tau}_{k}-\hat{\tau}_{k-1}|-(|\mathcal{B}^{(k)}|-2)d\}/2$, substituting into (\ref{eq:n1_n2}) and then maximising with respect to $d$, suggests that the probability of exceeding level $\epsilon$ is maximised when $d=d_{\min}$. We can thus use this setting of $n_{1},n_{2}$ to upper bound the error. The residual interval length $n_{1}+n_{2}$ either satisfies the two block analysis for $n_{1}+n_{2}<d_{\min}$ or otherwise. The stated result therefore follows in generality.
\end{proof}

\bibliographystyle{plainnat}
\bibliography{ejs_bib_final}

\end{document}